\documentclass[preprint,12pt]{elsarticle}
\usepackage{graphicx}
\usepackage[utf8]{inputenc}
\usepackage{palatino}
\usepackage{textcomp}
\usepackage{gensymb}
\usepackage{natbib}
\usepackage{amssymb,amsmath}
\usepackage{textcomp}
\usepackage{url}
\usepackage{amsmath}
\usepackage{amssymb}
\usepackage{amsthm}
\usepackage{setspace}
\usepackage{graphicx}
\usepackage{tikz}
\usetikzlibrary{chains,matrix,arrows,decorations.pathmorphing,backgrounds,positioning,fit,petri}
\usetikzlibrary{decorations.pathreplacing}
\usepackage{mathtools}
\usepackage{setspace}
\usepackage{natbib}
\usepackage{ mathrsfs }
\usepackage[retainorgcmds]{IEEEtrantools}
\usepackage{amssymb}
\usepackage{mathtools}
\usepackage{mathdots}

\newtheorem*{itheorem}{Theorem A}
\newtheorem*{jtheorem}{Theorem B}
\newtheorem{theorem}{Theorem}[section]

\newtheorem{lemma}[theorem]{Lemma}
\newtheorem{proposition}[theorem]{Proposition}
\newtheorem{definition}[theorem]{Definition}
\theoremstyle{definition} 
\theoremstyle{definition}
\newtheorem{conjecture}{Conjecture}[theorem]
\theoremstyle{definition}

\newtheorem{remark}[theorem]{Remark}

\theoremstyle{remark}

\usepackage{hyperref}
\hypersetup{colorlinks,%
citecolor=red,%
filecolor=black,%
linkcolor=cyan,%
urlcolor=black}
\begin{document}

\begin{frontmatter}


\title{Graded multiplicities in the exterior algebra of the little adjoint module}
\author{Ibukun Ademehin}
\address{School of Mathematics, University of Manchester, Manchester, M13 9PL, United Kingdom}
\ead{ibukun.ademehin@manchester.ac.uk}
\begin{abstract}
As a first application of the double affine Hecke algebra with unequal parameters on Weyl orbits to representation theory of semisimple Lie algebras, we find the graded multiplicities of the trivial module and of the little adjoint module in the exterior algebra of the little adjoint module of a simple Lie algebra $\,\mathfrak{g}\,$ with a non-simply laced Dynkin diagram. We prove that in type $\,B, C\,$ or $\,F\,$ these multiplicities can be expressed in terms of special exponents of positive long roots in the dual root system of $\,\mathfrak{g}.\,$
\end{abstract}


\end{frontmatter}


\section{Introduction}
Let $\,\mathfrak{g}$ be a simple Lie algebra of rank $\,r$ over $\, \mathbb{C}.\,$ Let $\,V_{\lambda}\,$ be the finite dimensional irreducible $\,\mathfrak{g}$-module with highest weight $\,\lambda.\,$ The graded multiplicity of $\,V_{\mu}\,$ in the exterior algebra $\,\bigwedge V_{\lambda}\,$ is the polynomial
\begin{equation*}
GM_{[\bigwedge V_{\lambda}:V_{\mu}]}(q)= \sum_{i\geq 0} d_i q^i,
\end{equation*}
where $\,d_i\,$ is the multiplicity of $\,V_{\mu}\,$ in $\,\bigwedge^i V_{\lambda}.\,$ When the exterior algebra is clearly identified we write $\,GM_{\mu}(q).\,$
\par Decomposing the exterior algebra of the adjoint module $\,\bigwedge \mathfrak{g}\,$ has been given some consideration since the last century following the more classical problem of decomposing the symmetric algebra $\,S\mathfrak{g}.\,$ Known results in $\,\bigwedge \mathfrak{g}\,$ include the following. Kostant \text{\cite{1Kos}} used the Hopf-Koszul-Samelson theorem, which asserts that the skew invariants $\, (\bigwedge \mathfrak{g})^{\mathfrak{g}}\,$ form an exterior algebra over the graded subspace of primitive invariants $\, \langle P_1, \cdots, P_r \rangle, \,$ to deduce the graded multiplicity of $\, (\bigwedge \mathfrak{g})^{\mathfrak{g}}\,$  expressed in terms of the exponents of $\, \mathfrak{g},\,$ which is the Poincar\'e polynomial of the De Rham cohomology of the Lie group $\,G$ associated with $\,\mathfrak{g}.\,$ Kostant also obtained in \text{\cite{1Kos}} the ungraded multiplicity of $\,\mathfrak{g}\,$ in $\,\bigwedge \mathfrak{g}.\,$
In \text{\text{\cite{3Baz}}} Bazlov proved Joseph's conjecture on the graded multiplicity of $\,\mathfrak{g} \,$ in $\, \bigwedge \mathfrak{g}.\,$ The formula for this graded multplicity suggests that the isotypic component of $\,\mathfrak{g}\,$ is a free module over the subring $\,\bigwedge \langle P_1, \cdots, P_{r-1} \rangle\,$ of the skew invariants. This indeed is the case, as was proved by De Concini, Papi and Procesi \text{\cite{4DPP}} who used the Chevalley transgression theorem to explicitly obtain basis vectors for the free module.
\par In \text{\cite{5Pan}} Panyushev obtained a classification of orthogonal irreducible $\,\mathfrak{g}$-modules $V$ whose skew invariants is again an exterior algebra (this classification includes the little adjoint module when the root system of $\,\mathfrak{g}\,$ is of type $\,B, C\,$ or $\,F\,$), and also proved the $\,\mathfrak{g}$-module isomorphism of $\, \bigwedge V\,$ to the reduced $\,Spin\,\text{ of } V.\,$ 
\par Besides these results not much is known about the decomposition of $\,\bigwedge V\,$ when $V$ is not the adjoint module. This motivates our study of the ring structure of the exterior algebra of the little adjoint module.
\subsection{The little adjoint module} Let $\,\mathfrak{g}\,$ be of type $\,B, C, F\,$ or $\,G.\,$ Let $\,V_{\theta_s}\,$ be the little adjoint module with highest weight the highest short root $\,\theta_s$ of $\,\mathfrak{g}\,$  and character $\,\chi_{\theta_s} =r_s+\sum_{\alpha \in R^s}e^{\alpha},\,$  where $\,r_s$ is the number of short simple roots of $\,R,\,$ the root system of $\,\mathfrak{g}\,$ and $\,R_s\,$ is the set of short roots in $R.$ Recall that in type $\,B_r,\,V_{\theta_s}\,$ is the standard $(2r+1)$-dimensional module of $\,so_{2r+1}(\mathbb{C}),\,\text{\cite[Chap. 19.4]{6FHa}}.\,$ To describe $\,V_{\theta_s}\,$ in type $\,C_r,\,$ let  $\,\{\omega_i\}_{i=1}^r\,$ be the set of fundamental weights of $\,\mathfrak{g}=sp_{2r}(\mathbb{C})\,$ and let $\,V_{\omega_{_1}}=\mathbb{C}^{2r}\,$ be the standard module of $\,sp_{2r}(\mathbb{C}).\,$ Let $\,\phi:\bigwedge^2 \mathbb{C}^{2r} \to \mathbb{C}\,$ be the contraction defined as $\,\phi(v\wedge w)= Q(v,w),\,$ where $\,Q\,$ is the skew-symmetric form associated with $\,sp_{2r}(\mathbb{C}).\,$ Then $\,V_{\theta_s}=V_{\omega_2}\,$ is the kernel of $\,\phi\,$ contained in $\,\bigwedge^2 \mathbb{C}^{2r}=V_{\theta_s} \oplus V_0,\, \text{\cite[Chap. 17.2]{6FHa}}.\,$ In type $\,F_4,\,$ $\,V_{\theta_s}\,$ is the standard module of the automorphism group of the $26$-dimensional $\,3\times 3\,$ Hermitian traceless Jordan octonion matrices, \text{\cite[Sect. 6.2.3]{7*}}. In type $\,G_2,\, V_{\theta_s}\,$ is the standard module of the Lie group $\,G_2,\,$ which is the automorphism group that preserves an alternating cubic form on the $7$-dimensional imaginary octonions, \text{\cite[Sect. 4]{8JCB}}.
\subsection{Main result} In the present paper, we prove the formulae for the graded multiplicities of the trivial module and of the little adjoint module in $\,\bigwedge V_{\theta_s},\,$ as Theorem~\text{\ref{T1}} and Theorem~\text{\ref{T2}}. We state these below.
\par
Let $\,R\,$ be of type $\,B, C\,$ or $\,F.\,$ Let $\,\{h_i\}_{i=1}^{r_s}\,$ be the set of special exponents of $\,(R_s^+)\,\check{}\,$ which form the partition dual to the partition arising from the positive long coroots in $\,R\,\check{}\,$ with respect to the special heights of these coroots (note that the coroots of short roots of $\,R$ have the long length in $\,(R_s^+)\,\check{},\,$ \text{\cite[Sect. 1.1]{9BFM}}). Then
\begin{itheorem}[Theorem~\text{\ref{T1}}]
The graded multiplicity of the trivial module $\,V_0\,$ in $\, \bigwedge V_{\theta_s}\,$ is given by 
\begin{equation}\label{z1}
 GM_0(q)=\prod_{i=1}^{r_s}(1+q^{2h_i+1}).
    \end{equation}
    \end{itheorem}
    
    \begin{jtheorem}[Theorem~\ref{T2}]
    The graded multiplicity of the little adjoint module in its exterior algebra $\, \bigwedge V_{\theta_s}\,$ is
\begin{equation}\label{z2}   
    GM_{\theta_s}(q)=\prod_{i=1}^{r_s-1}(1+q^{2h_i+1})\sum_{i=1}^{r_s}(q^{2h_i-(2h_1-1)}+q^{2h_i}).
    \end{equation}
\end{jtheorem}

When $\,R$ is of type $\,G_2\,$ the graded multiplicity of $\,V_0\,$ and  $\,V_{\theta_s}\,$ respectively in $\, \bigwedge V_{\theta_s}\,$ are
\begin{align} 
        GM_0(q)&=(1+q^3)(1+q^4),\label{z3}\\
    GM_{\theta_s}(q)&=(1+q^3)(q+q^2+q^3). \label{z4}
    \end{align}
    
\subsection{The method of proof} To obtain $\,GM_0(q)$ in formulae \text{\eqref{z1}} and \text{\eqref{z3}} we use Cherednik's inner product on characters \text{\cite{10Mac}} evaluated on $\, \chi_{\bigwedge V_{\theta_s}}\,$ and $\,\chi_0.\,$ This reduces the graded multiplicity $\,GM_0(q)\,$ to a ratio of polynomials in terms of the multi-parameter Poincar\'e polynomial $\,W(q^k)\,$ \text{\cite{11KLW}} and Cherednik's generalisation of the constant term of Macdonald's weight function $\,\Delta.\,$ In the generalised $\,W(q^k)\,$ and $\,ct(\Delta)\,$ we use different integer coefficients for the heights of short and long simple roots and of their coroots. This arises from the unique integer labelling $\,k(\alpha)\,$ on the different Weyl orbits of $\,R\,$ which occurs in the formulae $\,ct(\Delta)\,$ and $\,W(q^k)\,$ appearing in $\,\chi_{\bigwedge V_{\theta_s}}$ and in its inner products with $\,\chi_0.\,$ The positive integer $\,k(\alpha)\,$ relates the indeterminates $\,q$ and $\,t$ occurring in $\,W(q^k)\,$ and $\,ct(\Delta).\,$ In section $3\,$ where we prove the formula for $\,GM_0(q),\,$ we set $\,t_{\alpha}=q^{k(\alpha)},\,$ see \text{\cite{10Mac}}.
Cancellations occur in the ratio of $\,W(q^k)\,$ and $\,ct(\Delta)\,$ which reduce simplifying $\,GM_0(q)\,$ to a task of counting special heights of coroots in $\,(R_s^+)\,\check{}.\,$ We do this case by case when $R\,$ is of $\,B, C, F\,$ or $\,G.\,$ We note that in type $\,B, C\,$ or $\,F,\,$ $\,GM_0(q)\,$ can be expressed in terms of the partition dual to the partition arising from the coroots in $\,(R_s^+)\,\check{}\,$ with respect to the special heights of these coroots.
\begin{remark} The formulae for $\,GM_0(q)\,$ in type $\,B, C\,$ and $\,F\,$ are given without proof and with no reference to the special exponents in Panyushev's classification of orthogonal irreducible $\,\mathfrak{g}$-modules with an exterior algebra of skew invariants. See \text{\cite[Table 1]{5Pan}}.  
\end{remark}

\par To prove the formulae for $\,GM_{\theta_s}(q)\,$ in \text{\eqref{z2}} and \text{\eqref{z4}} we use the action of the operator $\,Y^{\theta{}\,\check{}}$ from the double affine Hecke algebra $\,\mathcal{H}\!\!\!\mathcal{H}_{q,\,t_{s,\,l}}$ on a subset of the group algebra $\,\mathbb{Q}_{q,\,t_{s,\,l}}[\mathcal{P}]\,$ generated by the integral weight lattice $\,\mathcal{P}\,$ over the field $\,\mathbb{Q}(q^{\pm \frac{1}{d}},\,t^{\pm 1}_{s,\,l}),\,$ where $\,\theta\,$ is the highest root of $\,R,\,$ and we follow Bazlov's treatment of $\,GM_{[\bigwedge \mathfrak{g}:\mathfrak{g}]}(q)\,$ in \text{\cite{3Baz}} (Bazlov in calculating $\,GM_{[\bigwedge \mathfrak{g}:\mathfrak{g}]}(q)\,$ used the label $\,k(\alpha)=1\,$ for all $\,\alpha\in R,\,$ but in our case of $\,GM_{\lambda}(q)\,$ in $\,\bigwedge V_{\theta_s}\,$ we maintain our unique integer labelling $\,k(\alpha)\,$ on the Weyl orbits of $\,R:$ we use $\,k(\alpha)=2\,$ for $\,\alpha\in R_s\,$ and $\,k(\alpha)=1\,$ for $\,\alpha \in R_l)\,$. The unique labelling $k(\alpha)\,$ on the $\,W$-orbits of $R\,$ necessitates our use of $\,Y^{\theta{}\,\check{}}\in \mathcal{H}\!\!\!\mathcal{H}_{q,\,t_{s,\,l}}\,$ and the double parameter $\,t_{s,\,l}\,$ for the indeterminate $\,t\,$ in $\,\mathbb{Q}_{q,\,t}[\mathcal{P}].\,$ The outcome of these is a non-trivial generalisation of Bazlov's results on the action of $\,Y^{\theta{}\,\check{}}\,$ on $\,\mathbb{Q}_{q,\,t}[\mathcal{P}]\,$ in \text{\cite{3Baz}}.  We use the action of $\,Y^{\theta\,\check{}}\,$ (which is unitary with respect to Cherednik's inner product $\,(\,\,,\,)\,$ on $\,\mathbb{Q}_{q,\,t_{s,\,l}}[\mathcal{P}])\,$ on $\,e^{\theta_s},\,$ as well as some combinatorial properties of a subset of $\,R\,$ associated with $\,Y^{\theta\,\check{}\,}$ to calculate $\,(e^{\alpha}\, ,\,1)\,$ for all $\,\alpha \in R_s.\,$ From these, the task of simplifying $\,GM_{\theta_s}(q)$ is reduced to counting the special heights of the positive short roots of $\,R\,$. We do this case by case for the different root systems. In type $\,B, C\,$ or $\,F\,$ again, we note that formula \text{\eqref{z2}} obtained for $\,GM_{\theta_s}(q)\,$, expressible in terms of the special exponents of $\,(R_s^+)\,\check{},\,$ suggests that the isotypic component of $\, V_{\theta_s}\,$ is a free module over a subring of the skew invariants in $\, \bigwedge V_{\theta_s}\,$. Motivated by De Concini, Papi and Procesi's result on the isotypic component of $\,\mathfrak{g}\,$ in $\,\bigwedge \mathfrak{g}\,$ \text{\cite{4DPP}}, we conclude the paper with a conjecture on the structure of the isotypic component of $\, V_{\theta_s}\,$ in the exterior algebra$\, \bigwedge V_{\theta_s}.\,$ Note that to obtain $\,GM_{\theta_s}(q)\,$ in Sect. \text{\ref{p}} we set
$\,t_{\alpha}=q^{\frac{-k(\alpha)}{2}},\,$ see \text{\cite[Sect. 4]{17Mac}} and \text{\cite{3Baz}}.

\subsection{Main notation}
Let $\,R\,$ be a root system of $\,\mathfrak{g}\,$ of type $\,B, C, F\,$ or $\,G,\,$ i.e. $\,R_s\neq\varnothing.\,$ We fix a basis of simple roots $\,\Phi =\{\alpha_i\}_{i=1}^{r} \,$ for $\,R\,$ and the corresponding set of fundamental weights $\,\{\omega_i\}_{i=1}^{r}.\,$\par
Let $\,\mathcal{Q}\,$ and $\,\mathcal{P}\,$ be the root and the integral weight lattice of $\,R\,$ spanned by simple roots $\,\{\alpha_i\}_{i=1}^{r}\,$ and fundamental weights $\,\{\omega_i\}_{i=1}^{r}\,$ respectively. Let $\,\mathcal{P}^+\,$ be the set of dominant integral weights of $\,\mathcal{P}\,$ and let $\,W\,$ be the Weyl group of $\,R\,$ generated by the simple reflections $\,s_i=s_{\alpha_i}.\,$\\
We use subscripts $s$ and $l$ to mark
objects related to short and long roots respectively. For instance, $\theta\,$ and $\,\theta_s\,$ are the highest root and the highest short root of $\,R\,$ respectively. We denote by the ordered pair $\,(k_s\,\,,\,k_l)\,$  the label $\,k(\alpha) \in \mathbb{Z}^{>0}\,$ of short and long roots respectively.

\section{Preliminaries}
Let $\chi_{\lambda} \in \mathbb{Z}[\mathcal{P}]\,$ be the character of $\,V_{\lambda}\,$ such that $\chi_{\lambda}=\sum_{\mu \in \mathcal{P}(V_{\lambda})}m_{\lambda}^{\mu} e^{\mu},\,$ where $\,\mathcal{P}(V_{\lambda})\,$ is the set of weights of $\,V_{\lambda}\,$ and $\,m_{\lambda}^{\mu}\,$ is the dimension of the weight space $\,V_{\lambda}^{\mu}\,$ of $\,\mu\,$ in $\,V_{\lambda}.\,$ The following is from \text{\cite[5.1]{10Mac}}. Let $\,\mathbb{Q}_{q,\,t}[\mathcal{P}]\,$ be the group algebra of $\,\mathcal{P}\,$ generated by formal exponentials $\,e^{\lambda},\,$ $\,\lambda\in\mathcal{P},\,$ over the field $\,\mathbb{Q}_{q,\,t}\,$ of rational functions in $\,q^{\pm \frac{1}{d}}\,$ and $\,t^{\pm 1},\,$ where $\,d=\mid\mathcal{P}/\mathcal{Q}\mid.\,$ Let $f=\sum_{\lambda \in \mathcal{P}}f_{\lambda}e^{\lambda}\,$ be an element of the group algebra $\,\mathbb{Q}_{q,\,t}[\mathcal{P}]\,$ and let the bar and $\,*\,$ involutions on $\,\mathbb{Q}_{q,\,t}[\mathcal{P}]\,$ be defined as
\begin{align*}
    \bar{}&: e^{\lambda} \mapsto e^{-\lambda},\quad q \mapsto q,\qquad  t \mapsto t,\\
    *&: e^{\lambda} \mapsto e^{-\lambda},\quad q \mapsto q^{-1},\quad t \mapsto t^{-1},
\end{align*}
extended by $\mathbb{Q}$-linearity over $\,\mathbb{Q}_{q,\,t}[\mathcal{P}].\,$
Let the symmetric and non-dege\-ne\-rate inner product due to Macdonald on $\,\mathbb{Q}_{q,\,t}[\mathcal{P}]\,$ be defined as
\begin{equation}\label{<>}
\langle f\,\,,\,h \rangle = \frac{1}{\mid W \mid} ct(f\bar{h}\nabla),
\end{equation}
where the constant term $\,ct(f\bar{h}\nabla)\in \mathbb{Q}_{q,\,t}\,$ of $\,f\bar{h}\nabla\,$ is the coefficient of $\,e^0\,$ in $\,f\bar{h}\nabla,\,$
\begin{equation*}
\nabla=\prod_{\alpha \in R} \prod_{i=0}^\infty \frac{1-q^i e^{\alpha}}{1-q^{k(\alpha)+i}e^{\alpha}},
\end{equation*}
and the integer labelling $\,k(\alpha)\,$ on the $\,W$-orbits of $\,R$ relates the indeterminates $\,q$ and $\,t,$ here as $\,t_{\alpha}=q^{k(\alpha)}\,$ (see \text{\cite[(5.1.1)]{10Mac}}). Observe that $\, \nabla\,$ becomes a finite product when $\,k(\alpha) \in \mathbb{Z}^{>0}:\,$
\begin{equation*}
    \nabla=\prod_{\alpha \in R} \prod_{i=0}^{k(\alpha)-1} (1-q^i e^{\alpha}).
\end{equation*}
We define the inner product due to Cherednik on $\,\mathbb{Q}_{q,\,t}[\mathcal{P}]\,$ as
\begin{equation*}
(f\,\,,\,h)=ct(fh^*\Delta),
\end{equation*}
where
\begin{equation*}
\Delta=\prod_{\alpha \in R} \prod_{i=0}^\infty \frac{(1-q^i e^{\alpha})(1-q^{i+1} e^{\alpha})}{(1-q^{k(\alpha)+i}e^{\alpha})(1-q^{k(\alpha)+i+1} e^{\alpha})}.
\end{equation*}
Here also, when $\,k(\alpha) \in \mathbb{Z}^{>0},\,$ $\, \Delta\,$ becomes a finite product:
\begin{equation*}
    \Delta=\prod_{\alpha \in R} \prod_{i=0}^{k(\alpha)-1} (1-q^i e^{\alpha})(1-q^{i+1} e^{\alpha}).
\end{equation*}
If $\,f,h\,$ are $\,W$-invariant in $\,\mathbb{Q}_{q,\,t}[\mathcal{P}],\,$ then
\begin{equation}\label{d}
    \langle f\,\,,\,\bar{h}^*\rangle=\frac{1}{W(q^k)}(f\,\,,\,h),
\end{equation}
where $\,W(q^k)\,$ is the Poincar\'e multi-parameter polynomial
\begin{equation}\label{e}
    W(q^k)=\sum_{w\in W} q^{\sum_{\alpha\in R(w)} k(\alpha)}= \prod_{\alpha\in R^+} \frac{1-q^{(\rho_k\,\,,\,\alpha\,\check{}\,)+k(\alpha)}}{1-q^{(\rho_k\,\,,\,\alpha\,\check{}\,)}},
\end{equation}
 \text{\cite[Sect. 2]{11KLW}}, $\,R(w)=\{\alpha\in R^+ \mid w(\alpha)\in R^-\}$ and $\,\rho_k\,$ is the double parameter special weight \text{\eqref{b}}.
 \par Let $\,\langle \,\,,\, \rangle_1\,$ denote $\,\langle \,\,,\, \rangle \,$ when $\,(k_s\,,\,k_l)=(1\,,\,1).\,$ The irreducible characters $\,\chi_{\lambda}, \, \lambda \in \mathcal{P}^+ \,$ are orthonormal with respect to $\,\langle \,\,,\, \rangle_1,\,$ \text{\cite[Sect. 5.3.15]{10Mac}} and therefore form an orthonormal basis for $\,\mathbb{Z}[\mathcal{P}].\,$
 Hence, if $\,J=\bigoplus_{i=0}^d J_i\,$ is a graded $\,\mathfrak{g}$-module with graded character $\,\chi_{_J}(q)=\sum_{i=0}^d q^i\chi_{_{J_i}},\,$ then the grad\-ed multiplicity of a $\,\mathfrak{g}$-module $\,V_{\lambda}\,$ in $\,J$ is
\begin{equation}\label{GM[]}
    GM_{\lambda}(q)=\langle \chi_{_J}(q)\,\,,\,\chi_{\lambda}\rangle_1.
\end{equation}
Recall $\,\chi_{\theta_s}=r_s+\sum_{\alpha \in R_s}e^{\alpha}\,$ \text{\cite[2.9]{5Pan}}, where $\,r_s=\#(R_s\cap \Phi).\,$ It is easy to show that
\begin{equation*}
    \chi_{\bigwedge V_{\theta_s}}(-q)=(1-q)^{r_s} \prod_{\alpha \in R_s} (1-qe^{\alpha})=(1-q)^{r_s} \frac{\nabla_{2,1}}{\nabla_1},
\end{equation*}
where in $\,\nabla_{2,1}\,$ and $\,\nabla_1$ we use the ordered pair $\,(k_s\,\,,k_l)=(2\,\,,1)\,$ and $\,(k_s\,\,,k_l)=(1\,\,,1)\,$ respectively. By \text{\eqref{<>}} and \text{\eqref{GM[]}} the graded multiplicity of $\,V_{\lambda}\,$ in $\,\bigwedge V_{\theta_s}\,$ evaluated at $\,-q\,$ becomes
\begin{equation}\label{a}
    GM_{\lambda}(-q)=(1-q)^{r_s} \langle 1\,\,,\,\chi_{\lambda}\rangle_{2,1}.
\end{equation}
Since $\,\chi_{\bigwedge V_{\theta_s}}(q) \in \mathbb{Z}_{q\,}[\mathcal{Q}],\,$ $\,GM_{\lambda}(q)=0\,$ when $\,\lambda \notin \mathcal{Q}.\,$ The problem of finding the graded multiplicities of the irreducible characters of $\,\bigwedge V_{\theta_s}\,$ is therefore reduced to calculating $\,\langle 1 \,\,,\,\chi_{\lambda}\rangle_{2,1},\,$ for $\,\lambda \in \mathcal{P}^+\cap \mathcal{Q}.\,$ In the next sections we will calculate $\,GM_{\lambda}(q)\,$ for the two smallest dominant elements of $\,\mathcal{P}^+ \cap \mathcal{Q},\,$ i.e $\,\lambda = 0$ and $\,\theta_s,\,$ to prove our results \text{\eqref{z1} -- \eqref{z4}}.

\section{Calculating \texorpdfstring{$\,GM_0(q)\,$}{Lg}}
By \text{\eqref{<>}}, \text{\eqref{d}} and \text{\eqref{a}} the graded multiplicity of the trivial module $\,V_0\,$ in $\,\bigwedge V_{\theta_s}\,$ is given by
\begin{equation}\label{c}
    GM_0(-q)=\frac{(1-q)^{r_s} ct(\Delta_{2,1})}{\mid W\mid W(q^k)}.
\end{equation}
To simplify this we consider the following.
\subsection{The double parameter special weight \texorpdfstring{$\,\rho_k\,$}{Lg}}
Set $\,\rho_s= \frac{1}{2}\sum_{\alpha \in R_s^+} \alpha=\sum_{\alpha_i \in R_s} \omega_i\,\,$ and $\,
    \rho_l=\frac{1}{2}\sum_{\alpha \in R_l^+} \alpha= \sum_{\alpha_i \in R_l} \omega_i.\,$ See \text{\cite[(3.1.1)]{12Che}}.
The \textbf{double parameter special weight} $\,\rho_k\,$ is defined as
\begin{equation}\label{b}
    \rho_k=k_s \rho_s + k_l \rho_l=\frac{1}{2}\sum_{\alpha \in R^+} k(\alpha)\alpha=\sum_{i=1}^{r} k(\alpha_i)\omega_i.
\end{equation}
See \text{\cite[(3.2.2)]{12Che}}. We extend $\,(\rho_k\,,\,\,)\,$ over $\,\mathcal{Q}^+\,\check{}\,$ by $\,\mathbb{Z}$-linearity and say $\,\alpha\,\check{} \in \mathcal{Q}^+\,\check{}\,$ has special height $\,(\rho_k\,\,,\alpha\,\check{}\,).\,$
\par
The following theorem due to Cherednik \text{\cite[(3.3.2)]{12Che}} provides a generalisation of Macdonald's constant term of $\,\Delta\,$ by allowing different values for $\,k(\alpha)\,$ on the different $\,W $-orbits of $R.\,$
\begin{theorem}[Macdonald's Constant Term]  
The constant term of Macdonald's weight function $\,\Delta\,$ is
\begin{equation*}
    ct(\Delta) = \prod_{\alpha \in R^+} \prod_{i=1}^\infty \frac{(1-q^{(\rho_k\,\, , \, \alpha\,\check{}\,)\,+\,i})^2}{(1-t_\alpha q^{(\rho_k\,\, , \, \alpha\,\check{}\,)\,+\,i})(1-t^{-1}_\alpha q^{(\rho_k\,\, , \, \alpha\,\check{}\,)\,+\,i})},
\end{equation*}
where $\,t_{\alpha}=q^{k(\alpha)}.\,$
\end{theorem}

\par Observe that when $\,k(\alpha) \in \mathbb{Z}^{\geq 0},\, ct(\Delta)\,$ becomes a finite product:
\begin{equation}\label{z5}
    ct(\Delta)=\prod_{\alpha \in R^+} \prod_{i=1}^{k(\alpha)} \frac{1-q^{(\rho_k\,\, , \, \alpha\,\check{}\,)\,+\,i}}{1-q^{(\rho_k\,\, , \, \alpha\,\check{}\,)\,+1-\,i}}.
\end{equation}
By \text{\eqref{e}}, \text{\eqref{b}}, \text{\eqref{z5}} and using $\,(k_s\,\,,k_l)=(2\,\,,1),\,$ $\,GM_0(-q)\,$ in \text{\eqref{c}} simplifies to
\begin{equation}\label{f}
   GM_0(-q)=(1-q)^{r_s} \prod_{\alpha \in R_s^+} \frac{1-q^{(\rho_k\,\, , \, \alpha\,\check{}\,)\,+\,1}}{1-q^{(\rho_k\,\, , \, \alpha\,\check{}\,)\,-\,1}}. 
\end{equation}
\par Let $\,Z\,\check{}\subseteq (R_s^+)\,\check{}.\,$ We denote by $\,H_{Z\,\check{}}(n)\,$ the set $\,\{\alpha\,\check{} \in Z\,\check{} \mid (\rho_k\,\, , \, \alpha\,\check{}\,)=n\}.\,$ Let $\,h_{Z\,\check{}}(n)=\#H_{Z\,\check{}}(n).\,$ Observe that $\,\alpha\,\check{} \in H_{R_s^+\,\check{}}(n-1)\,$ gives the factor $\,1-q^n\,$ in \text{\eqref{f}}, while $\,\alpha\,\check{} \in H_{R_s^+\,\check{}}(n+1)\,$ gives $\,(1-q^n)^{-1}.\,$ Therefore, since $\,\theta_s\check{}\,$ is the highest coroot in $\,(R_s^+)\,\check{}\,$ \text{\cite[Lemma 5.1.4]{13PR}}, $\,GM_0(-q)\,$ becomes
\begin{equation}\label{g}
        GM_0(-q)=(1-q)^{r_s}\prod_{n=1} ^{(\rho_k\, , \, \theta_s\check{}\,) +1} (1-q^n)^{h_{R_s^+\,\check{}}(n-1)-h_{R_s^+\,\check{}}(n+1)}.
\end{equation}
Simplifying $\,GM_0(q)\,$ is therefore reduced to a task of counting the special heights of positive long roots in the dual root system $R\,\check{}\,$ (recall that the coroots of the roots in $\,R_s$ have the long length in $\,R\,\check{},\,$ \text{\cite[Sect. 1.1]{9BFM}}). We do this case by case when $\, R \,$ is of type $\, B, C, F\,$ or $\,G.\,$

\subsubsection{Type \texorpdfstring{$\,B_r\,$}{Lg}}
Recall that the sets of positive short roots and their coroots in type $\,B_r\,$ are
\begin{align*}
R_s^+&=\{\sum_{i\leq m\leq r} \alpha_m, \,\, 1\leq i \leq r\},\\
    (R_s^+)\,\check{}&=\{\sum_{i\leq m < r} 2\alpha\,\check{}_m +\alpha\,\check{}_r,\, \, 1\leq i \leq r\},
\end{align*}
where $\,\Phi_l=\{\alpha_1,\cdots,\alpha_{r-1}\},\,$ $\,\Phi_s=\{\alpha_r\}\,$ and the highest coroot $\,\theta_s\,\check{}=2\alpha\,\check{}_1+\cdots+2\alpha\,\check{}_{r-1}+\alpha\,\check{}_r.\,$  See \text{\cite[Plate II]{14Bou}}.
With $\,(k_s\,\,,\,k_l)=(2\,\,,\,1)\,$ in \text{\eqref{a}}, $\,(\rho_k\, , \, \sum_{i\leq m < r} 2\alpha\,\check{}_m +\alpha\,\check{}_r)=2(r-i+1).\,$ Hence, $\,h_{R_s^+\,\check{}}(2i)=1,\,1\leq i \leq r,\,$ otherwise, $\,h_{R_s^+\,\check{}}(n)=0.\,$ Therefore, $\,h_{R_s^+\,\check{}}(n-1)=h_{R_s^+\,\check{}}(n+1)\,$
for $\,2\leq n \leq 2r.\,$ We substitute this in \text{\eqref{g}} and obtain in type $\,B_r\,$
\begin{equation}\label{j}
    GM_0(q)=1+q^{2r+1}.
\end{equation}
\begin{remark}
Formula \text{\eqref{j}} implies that the only skew invariants of the Lie\- group $\,SO_{2r+1}(\mathbb{C})\,$ in its standard module are the scalars and the volume form.
\end{remark}
\subsubsection{Type \texorpdfstring{$\,C_r\,$}{Lg}}
Recall that the set of the positive short roots in type $\,C_r\,$ is given as
\begin{align*}
    R_s^+&=J \cup K\nonumber\\
   & =\{\sum_{i\leq m<j} \alpha_m \} \, \bigcup \, \{\sum_{i\leq m < j} \alpha_m+\sum_{j\leq m < r}2\alpha_m + \alpha_r\}, \quad 1\leq i<j \leq r;
\end{align*}
whereas their coroots are
\begin{align}\label{z18}
(R_s^+)\,\check{}&=J\,\check{} \cup K\,\check{},\nonumber\\
    J\,\check{}&=\{\sum_{i\leq m < j} \alpha\,\check{}_m, \quad 1\leq i<j \leq r\},\nonumber\\
    K\,\check{}&=\{\sum_{i\leq m < j} \alpha\,\check{}_m+ \sum_{j \leq m \leq r}2\alpha\,\check{}_m, \quad 1\leq i<j \leq r \};
\end{align}
where $\,\Phi_s=\{\alpha_1,\cdots,\alpha_{r-1}\},\,$ $\,\Phi_l=\{\alpha_r\}\,$ and $\,\theta_s\,\check{}=\alpha\,\check{}_1+2\alpha\,\check{}_2+\cdots+2\alpha\,\check{}_{r-1}+2\alpha\,\check{}_r.\,$ See \text{\cite[Plate III]{14Bou}}.
\par Let $\,Z\,\check{}\subseteq (R_s^+)\,\check{}.\,$ We denote $\,h_{Z\,\check{}}(n)\,$ by $\,h^o_{Z\,\check{}}(n)\,$ (respectively $\,h^e_{Z\,\check{}}(n))\,$ when the rank $\,r$ of $\,\mathfrak{g}\,$ is odd (respectively even). By \text{\eqref{z18}} we have
\begin{equation*}
H_{J\,\check{}}(2p)=\{\alpha\check{}_i+\alpha\check{}_{i+1}+\cdots+\alpha\check{}_{i+p-1},\quad 1\leq p \leq r-1,\,\, 1\leq i \leq r-p\}.
\end{equation*}
Since $\,(k_s\, , \,k_l)=(2\,\,,1),\,$
\begin{equation}\label{z14}
h_{J\,\check{}}(2p)=r-p, \qquad 1\leq p \leq r-1.
\end{equation}
Let $\,\beta_{i, j}=\sum_{i\leq m < j} \alpha\check{}_m +\sum_{j\leq m \leq r} 2\alpha\check{}_m \in K\,\check{} ,\,\, 1\leq i<j \leq r.\,$ Then $\,(\rho_k\,,\,\beta_{i, j})=2(j-i)+4(r-j)+2
\,$ and $\,(\rho_k\,,\,\beta_{i, j}) \in \{4,8, \cdots, 4r-4\}.\,$
\par Observe that when $\,r\,$ is odd and $\,p=1,3,\cdots,r-2,\,$ respectively when $\,r\,$ is even and $\,p=1,3,\cdots,r-3,\,$ we have
\begin{align*}
        H_{K\,\check{}}(2p+2)=&\{\alpha\check{}_{r-\frac{p+1}{2}}+2\alpha\check{}_{r-\frac{p-1}{2}}+\cdots+2\alpha\check{}_{r},\,\alpha\check{}_{r-\frac{p+3}{2}}+\cdots+\alpha\check{}_{r-\frac{p-1}{2}}+\nonumber\\
        &2\alpha\check{}_{r-\frac{p-3}{2}}+\cdots+2\alpha\check{}_{r},\cdots,\, \alpha\check{}_{r-p}+\cdots+\alpha\check{}_{r-1}+2\alpha\check{}_{r}\},\\
       H_{K\,\check{}}(2p+4)=&\{\alpha\check{}_{r-\frac{p+3}{2}}+\cdots+\alpha\check{}_{r-1}+2\alpha\check{}_{r}, \,\alpha\check{}_{r-\frac{p+5}{2}}+\cdots+\alpha\check{}_{r-\frac{p-1}{2}}+\nonumber\\&
       2\alpha\check{}_{r-\frac{p-3}{2}}+\cdots+2\alpha\check{}_{r},\cdots,\,\alpha\check{}_{r-(p+1)}+\cdots+\alpha\check{}_{r-1}+2\alpha\check{}_{r}\},\\ H_{K\,\check{}}(4r-2p-4)=&\{\alpha\check{}_{1}+\cdots+\alpha\check{}_{p+1}+2\alpha\check{}_{p+2}+\cdots+2\alpha\check{}_{r},\,\alpha\check{}_{2}+\cdots+\alpha\check{}_{p}+\nonumber\\
        &2\alpha\check{}_{p+1}+\cdots+2\alpha\check{}_{r},\cdots,\, \alpha\check{}_{\frac{p+1}{2}}+\alpha\check{}_{\frac{p+3}{2}}+2\alpha\check{}_{\frac{p+5}{2}}+\cdots+
        \nonumber\\
        &2\alpha\check{}_{r}\},\nonumber\\
       H_{K\,\check{}}(4r-2p-2)=&\{\alpha\check{}_{1}+\cdots+\alpha\check{}_{p}+2\alpha\check{}_{p+1}+\cdots+2\alpha\check{}_{r},\,\alpha\check{}_{2}+\cdots+\alpha\check{}_{p-1}+\nonumber\\
        &2\alpha\check{}_{p}+\cdots+2\alpha\check{}_{r},\cdots,\, \alpha\check{}_{\frac{p+1}{2}}+2\alpha\check{}_{\frac{p+3}{2}}+\cdots+2\alpha\check{}_{r}\}, \\
        H^e_{K\,\check{}}(2r)=&\{\alpha\check{}_{1}+\cdots+\alpha\check{}_{r-1}+2\alpha\check{}_{r},\, \alpha\check{}_{2}+\cdots+\alpha\check{}_{r-2}+2\alpha\check{}_{r-1}+2\alpha\check{}_{r},\nonumber\\& \cdots,\, \alpha\check{}_{\frac{r}{2}}+2\alpha\check{}_{\frac{r+2}{2}}\cdots+2\alpha\check{}_{r}\}.
\end{align*}
Hence, 
\begin{align}\label{z15}
    h^o_{K\,\check{}}(n)=h^e_{K\,\check{}}(n)&=\frac{p+1}{2}\quad\text{and}\nonumber\\
     h^e_{K\,\check{}}(2r)&=\frac{r}{2}
     \end{align}
whenever $\,n\in \{2p+2, 2p+4, 4r-2p-2, 4r-2p-4\},\,r\in 2\mathbb{Z}+1 \,\text { and }\, p=1,3,\cdots,r-2,\,$ respectively $\,r\in 2\mathbb{Z} \,\text { and }\, p=1,3,\cdots,r-3.\,$
\par Table \text{\ref{(r_k,a)CrZ}} is obtained from \text{\eqref{z14}} and \text{\eqref{z15}}. The reader can easily check that the cardinalities in columns $\,2$--$6\,$ in Table \text{\ref{(r_k,a)CrZ}} match the cardinalities of $\,J\,\check{},\,K\,\check{}\,$ and $\,R_s^+\,\check{}\,$ in \text{\eqref{z18}}.

\begin{table}[htb]
\centering
\begin{tabular}{ |c|c|c|c|c|c| }
\hline
$n$ & $h_{J\,\check{}}(n)$ & $h^o_{K\,\check{}}(n)$ & $h^e_{K\,\check{}}(n)$ & $h^o_{R_s^+\,\check{}}(n)$ & $h^e_{R_s^+\,\check{}}(n)$\\
\hline
$2$ & $ r-1$ & $0$ & $0$ & $r-1$ & $r-1$\\
\hline
$4$ & $r-2$ & $1$ & $1$ & $r-1$ & $r-1$\\
\hline
$6$ & $r-3$ & $1$ & $1$ & $r-2$ & $r-2$\\
\hline
$8$ &$ r-4$ & $2$ & $2$ & $r-2$ & $r-2$\\
\hline
$\vdots$ & $\vdots$ & $\vdots$ & $\vdots$ & $\vdots$ & $\vdots$\\
\hline
$2r- 6$ & $3$ &$\frac{r-3}{2}$ & $\frac{r-4}{2}$& $\frac{r+3}{2}$& $\frac{r+2}{2}$\\
\hline
$2r- 4$ & $2$ &$\frac{r-3}{2}$ & $\frac{r-2}{2}$& $\frac{r+1}{2}$& $\frac{r+2}{2}$\\
\hline
$2r-2$ & $1$ &$\frac{r-1}{2}$ & $\frac{r-2}{2}$& $\frac{r+1}{2}$& $\frac{r}{2}$\\
\hline
$2r$ & $0$ &$\frac{r-1}{2}$ & $\frac{r}{2}$& $\frac{r-1}{2}$& $\frac{r}{2}$\\
\hline
$2r+2$ & $0$ &$\frac{r-1}{2}$ & $\frac{r-2}{2}$& $\frac{r-1}{2}$& $\frac{r-2}{2}$\\
\hline
$2r+4$ & $0$ &$\frac{r-3}{2}$ & $\frac{r-2}{2}$& $\frac{r-3}{2}$& $\frac{r-2}{2}$\\
\hline
$\vdots$ & $\vdots$ & $\vdots$ & $\vdots$ & $\vdots$ & $\vdots$\\
\hline
$4r-6$ &$0$ & $1$ & $1$& $1$ & $1$\\
\hline
$4r-4$ & $0$ & $1$ & $1$ & $1$ & $1$\\
\hline
\end{tabular}

\caption{The multiplicity of the special heights of the positive long coroots in type $\,C_r\,$ }\label{(r_k,a)CrZ}
\end{table}

Observe that for any rank $\,r\,$
\begin{equation}\label{i}
    h_{R_s^+\,\check{}}(n-1)=h_{R_s^+\,\check{}}(n+1)+1 \quad\text{ when }\, n=4p+1, \, p = 1,\cdots, r-1,
    \end{equation} 
   otherwise, when $\,n\neq1\,$
    \begin{equation}\label{z6}
    h_{R_s^+\,\check{}}(n-1)=h_{R_s^+\,\check{}}(n+1).
\end{equation}
We substitute \text{\eqref{i}} and \text{\eqref{z6}} into \text{\eqref{g}} and $GM_0(q)\,$ in type $\,C_r\,$ simplifies to
\begin{equation}\label{m}
    GM_0(q)=\prod_{p=1}^{r-1}(1+q^{4p+1}).
\end{equation}

\subsubsection{Type \texorpdfstring{$\,F_4\,$}{Lg}}\label{zz}
Let $\,\sigma_k=\sum_{i =1}^r k(\alpha_i)\omega_i\,\check{},\,$ where $\,\omega\,\check{}_i\,$ are the coweights of the fundamental weights of $\,\mathcal{P}\,$ dual to the simple roots of $R.\,$ With $\,(k_s\,\,,\,k_l)=(2\,\,,\,1),\,$ $\,(\sigma_k\,\,,\,\alpha)= ht_l\, \alpha+2ht_s\, \alpha,\,$ where $\, ht_m\,\alpha = \sum_{\alpha_i \in \Phi_m} \lambda_i,\,\,m\in \{s,\,l\},\,$ when $\alpha=\sum_{i=1}^r\lambda_i\alpha_i.\,$ 
\par From \text{\cite[Sect. VI: 4.9, Plate VIII]{14Bou}} the positive short roots and their coroots in type $\,F_4\,$ are as in Table~\text{\ref{n}}.

\begin{table}[htb]
    \centering
\begin{tabular}{ |c|c|c|c| }
\hline
$ \alpha \in R^+_s $ & $ \alpha\,\check{} $ & $ (\sigma_k \, , \, \alpha)$ & $ (\rho_k \, , \, \alpha\,\check{}\,)$\\
 \hline
 $\alpha_3$ & $\alpha\,\check{}_3$ & $2$ & $2$\\
 \hline
 $\alpha_4$ & $\alpha\,\check{}_4$ & $2$ & $2$\\
 \hline
 $\alpha_2 + \alpha_3$ & $2\alpha\,\check{}_2 + \alpha\,\check{}_3$ & $3$ & $4$\\
 \hline
 $\alpha_3 + \alpha_4$ & $\alpha\,\check{}_3 + \alpha\,\check{}_4$ & $4$ & $4$\\
 \hline
 $\alpha_1 + \alpha_2+\alpha_3$ & $2\alpha\,\check{}_1 + 2\alpha\,\check{}_2 +\alpha\,\check{}_3$ & $4$ & $6$\\
 \hline
 $\alpha_2 + \alpha_3+\alpha_4$ & $2\alpha\,\check{}_2 +\alpha\,\check{}_3+ \alpha\,\check{}_4$ & $5$ & $6$\\
 \hline
  $\alpha_1 +\alpha_2 +\alpha_3 +\alpha_4$ & $2\alpha\,\check{}_1 +2\alpha\,\check{}_2 +\alpha\,\check{}_3 +\alpha\,\check{}_4$ & $6$ & $8$\\
 \hline
 $\alpha_2 +2\alpha_3 +\alpha_4$ & $2\alpha\,\check{}_2 +2\alpha\,\check{}_3 +\alpha\,\check{}_4$ & $7$ & $8$\\
 \hline
 $\alpha_1 +\alpha_2 + 2\alpha_3 +\alpha_4$ & $2\alpha\,\check{}_1 +2\alpha\,\check{}_2 +2\alpha\,\check{}_3 + \alpha\,\check{}_4$ & $8$ & $10$\\
 \hline
 $\alpha_1 +2\alpha_2 + 2\alpha_3 +\alpha_4$ & $2\alpha\,\check{}_1 + 4\alpha\,\check{}_2 +2\alpha\,\check{}_3 +\alpha\,\check{}_4$ & $9$ & $12$\\
 \hline
 $\alpha_1 + 2\alpha_2+3\alpha_3+\alpha_4$ & $2\alpha\,\check{}_1 + 4\alpha\,\check{}_2 +3\alpha\,\check{}_3 +\alpha\,\check{}_4$ & $11$ & $14$\\
 \hline
 $\alpha_1+2\alpha_2 + 3\alpha_3+2\alpha_4$ & $2\alpha\,\check{}_1 + 4\alpha\,\check{}_2 +3\alpha\,\check{}_3 +2\alpha\,\check{}_4$ & $13$ & $16$\\\hline
\end{tabular}
\caption{The special heights of the short roots and of their coroots in type $F_4$} \label{n}
\end{table}

We substitute Table \text{\ref{n}} into \text{\eqref{g}} and obtain $\,GM_0(q)\,$ in type $\,F_4\,$ as
\begin{equation}\label{u}
        GM_0(q)=(1+q^9)(1+q^{17}).
        \end{equation}

\subsubsection{Type \texorpdfstring{$\,G_2\,$}{Lg}}
The positive short roots and their coroots in the root system of type $G_2\,$ are as given in Table~\text{\ref{h}}. See \text{\cite[Sect. VI: 4.13, Plate IX]{14Bou}}.
\begin{table}[htb]
\centering
\begin{tabular}{ |c|c|c|c| }
\hline
$ \alpha \in R_s^+ $ & $ \alpha\,\check{} $ & $ (\sigma_k \, , \, \alpha)$ & $ (\rho_k \, , \, \alpha\,\check{}\,)$\\
 \hline
 $\alpha_1$ & $\alpha\,\check{}_1$ & $2$ & $2$\\
 \hline
 $\alpha_1 + \alpha_2$ & $\alpha\,\check{}_1 + 3\alpha\,\check{}_2$ & $3$ & $5$\\
 \hline
 $2\alpha_1 + \alpha_2$ & $2\alpha\,\check{}_1 + 3\alpha\,\check{}_2$ & $5$ & $7$\\
 \hline
\end{tabular}
    \caption{The special heights of the positive short roots and the positive long coroots in type $\,G_2\,$}\label{h}
    \end{table}
    
We substitute  Table \text{\ref{h}} into \text{\eqref{g}} and get $\,GM_0(q)\,$ in type $\,G_2\,$ as
\begin{equation}\label{w}
        GM_0(q)=(1+q^3)(1+q^{4}).
\end{equation}
\par The reader can easily check that the degree of $\,GM_0(q)\,$ for the different root systems is indeed $\,\dim V_{\theta_s}=r_s+\#R_s.\,$ This is exactly what we expect since if $\dim V=d,\,$ then $\,\bigwedge V= \bigwedge^0 V\oplus \cdots \oplus \bigwedge^d V\,$ and $\,\bigwedge^d V \subset (\bigwedge V)^{\mathfrak{g}}\,$ since $\,\bigwedge^d V$ is isomorphic to the trivial module of $\,\mathfrak{g}.$
\par The factorisation of $\,GM_0(q)\,$ obtained with the odd powers of $\,q\,$ when $\,R\,$ is of type $\,B, C\,$ or $\,F\,$ suggests that the skew invariants $\,(\bigwedge V_{\theta_s})^{\mathfrak{g}}\,$ form an exterior algebra over primitive invariant generators in $\,\bigwedge V_{\theta_s}\,$ in these cases of $\,R.\,$ This indeed is the case as $\,V_{\theta_s}\,$ in type $\,B, C\,$ or $\,F\,$ is included in Panyushev's  classification of orthogonal $\,\mathfrak{g}$-modules with an exterior algebra of skew invariants; see \text{\cite[Table 1]{5Pan}}. In type $\,G_2,\,$ let $\,(\bigwedge V_{\theta_s})^{\mathfrak{g}}=span\, \{T_0,T_3,T_4,T_7\,\},\,$ where $\,T_i\,$ is the invariant generator of degree $i\,$ in $\,\bigwedge V_{\theta_s}.\,$ By \text{\cite[Lemma 1.3]{5Pan}} $\,(\bigwedge V_{\theta_s})^{\mathfrak{g}}\,$ does not form an exterior algebra over primitive generators. Indeed, since $\,T_0,T_4,T_7 \in Z(\bigwedge V_{\theta_s}),\,$ the centre of $\,\bigwedge V_{\theta_s},\,$ the algebra of skew invariants $\,(\bigwedge V_{\theta_s})^{\mathfrak{g}}\,$ is commutative when $\,R\,$ is of type $\,G_2.\,$

\subsection{\texorpdfstring{$\,GM_0(q)\,$}{Lg} and the special exponents of 
\texorpdfstring{$\,(R_s^+)\,\check{}\,$}{Lg}}
In what follows we define the special exponents of $\,(R_s^+)\,\check{}\,$ when $\,R\,$ is of type $\,B, C, F\,$ or $\,G\,$ and give as one of our main results: $\,GM_0(q)\,$ expressed in terms of these special exponents.
\begin{definition}
Let $\,R\,$ be of type $\,B, C, F\,$ or $\,G\,$. The \textbf{special exponents} of $\,(R_s^+)\,\check{},\, h_1 \leq \cdots \leq h_{r_s}\,$ are the partition dual to the partition formed by the positive long roots with respect to their special height in the dual root system $\,R\,\check{}.\,$
\end{definition} 
The list $\,\{h_i\}_{i=1}^{r_s}\,$ in each of these types of $\,R$ is given as
\begin{equation*}
    B_r : \{r\}, \quad
    C_r : \{2i\}_{i=1}^{r-1}, \quad
    F_4 :\{4,8\},\quad
    G_2 :\{3\}.
\end{equation*}
See Figure \text{\ref{z12}}.

\begin{theorem}\label{T1}
Let $\,R\,$ be of type $\,B, C$ or $F\,$. Let $\, h_1,\cdots,h_{r_s}\,$ be the special exponents of $\,(R_s^+)\,\check{}.\,$ The graded multiplicity of the trivial module $\,V_0\,$ in $\,\bigwedge V_{\theta_s}\,$ is given by
 \begin{equation*}
 GM_0(q)=(1+q^{2h_1+1})\cdots(1+q^{2h_{r_s}+1}).
 \end{equation*}
 When $\,R\,$ is of type $\,G_2, \, GM_0(q)\,$ in $\,\bigwedge V_{\theta_s}\,$ is
\begin{equation*}
 GM_0(q)=(1+q^{h_1})(1+q^{h_{1}+1}).  
 \end{equation*}
\end{theorem}

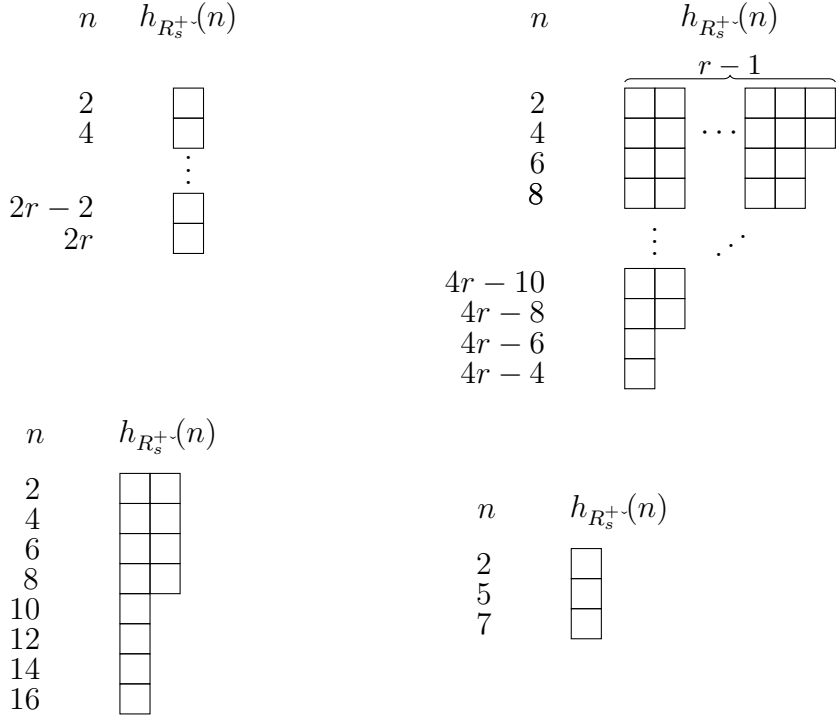
\begin{figure}[htb]
\begin{tikzpicture}
 [inner sep=2mm,
place/.style={rectangle,draw}]
\node (a1) at ( 0,2.0) [place] {};
\node [left] at (a1.west) {$2\quad\quad$};
\node (b1) at ( 0,1.6) [place] {};
\node [left] at (b1.west) {$4\quad\quad$};
\node (c1) at ( 0,0.6) [place] {};
\node [left] at (c1.west) {$2r-2\quad\quad$};
\node (d1) at ( 0,0.2) [place] {};
\node [left] at (d1.west) {$2r\quad\quad$};
\node [below] at (0,1.75) {$\vdots$};
\node (e1)  at (0,3.1) {$h_{R_s^+\,\check{}}(n)$};
\node (f1) [left] at (e1.west) {$n\,\,$};

\node (a) at ( 6,2.0) [place] {};
\node (g) at ( 6.4,2.0) [place] {};
\node (e) [left] at (a.west) {$2\quad\quad$};
\node (b) at ( 6,1.6) [place] {};
\node (h) at ( 6.4,1.6) [place] {};
\node [left] at (b.west) {$4\quad\quad$};
\node (i) at ( 6,1.2) [place] {};
\node (j) at ( 6.4,1.2) [place] {};
\node (k) [left] at (i.west) {$6\quad\quad$};
\node (l) at ( 6,0.8) [place] {};
\node (m) at ( 6.4,0.8) [place] {};
\node [left] at (m.west) {$8\quad\quad\quad$};
\node [right] at (6.6,1.6) {$\cdots$};

\node (p) at ( 7.6,2.0) [place] {};
\node (q) at ( 8,2.0) [place] {};
\node (x) at ( 8.4,2.0) [place] {};
\node (y) at ( 8.4,1.6) [place] {};
\node (r) at ( 7.6,1.6) [place] {};
\node (s) at ( 8,1.6) [place] {};
\node (t) at ( 7.6,1.2) [place] {};
\node (u) at ( 8,1.2) [place] {};
\node (v) at ( 7.6,0.8) [place] {};
\node (w) at ( 8,0.8) [place] {};
\node [left] at (m.west) {$8\quad\quad\quad$};
\node [below] at (6.2,0.8) {$\vdots$};
\node [below] at (7.2,0.75) {$\iddots$};

\node (c) at ( 6,-0.4) [place] {};
\node (z) at ( 6.4,-0.4) [place] {};
\node [left] at (c.west) {$4r-10\quad\quad$};
\node (d) at ( 6,-0.8) [place] {};
\node (aa) at ( 6.4,-0.8) [place] {};
\node [left] at (d.west) {$4r-8\quad\quad$};
\node (n) at ( 6,-1.2) [place] {};
\node [left] at (n.west) {$4r-6\quad\quad$};
\node (o) at ( 6,-1.6) [place] {};
\node [left] at (o.west) {$4r-4\quad\quad$};
\node (e)  at (7.2,3.1) {$h_{R_s^+\,\check{}}(n)$};
\node (f) [left] at (e.west) {$n\qquad\,\quad$};
\node [font=\small] at (7.2,2.5) {$r-1$};

\draw [decorate,decoration={brace},xshift=0pt,yshift=0pt] (5.8,2.28) -- (8.6,2.28);
\end{tikzpicture}

\begin{tikzpicture}
 [inner sep=2mm,
place/.style={rectangle,draw}]
\node (a) at ( 0,2.0) [place] {};
\node (g) at ( 0.4,2.0) [place] {};
\node (e) [left] at (a.west) {$2\quad\quad$};
\node (b) at ( 0,1.6) [place] {};
\node (h) at ( 0.4,1.6) [place] {};
\node [left] at (b.west) {$4\quad\quad$};
\node (i) at ( 0,1.2) [place] {};
\node (j) at ( 0.4,1.2) [place] {};
\node (k) [left] at (i.west) {$6\quad\quad$};
\node (l) at ( 0,0.8) [place] {};
\node (m) at ( 0.4,0.8) [place] {};
\node [left] at (m.west) {$8\quad\quad\quad$};
\node (c) at ( 0,0.4) [place] {};
\node [left] at (c.west) {$10\quad\quad$};
\node (d) at ( 0,0.0) [place] {};
\node [left] at (d.west) {$12\quad\quad$};
\node (n) at ( 0,-0.4) [place] {};
\node [left] at (n.west) {$14\quad\quad$};
\node (o) at ( 0,-0.8) [place] {};
\node [left] at (o.west) {$16\quad\quad$};
\node (e)  at (0.4,2.7) {$\,h_{R_s^+\,\check{}}(n)$};
\node (f) [left] at (e.west) {$n\,\quad$};

\node (m2)  at (6.4,1.7) {$\,h_{R_s^+\,\check{}}(n)$};
\node (n2) [left] at (m2.west) {$n\,\quad$};
\node (a2) at ( 6,1.0) [place] {};
\node (e2) [left] at (a2.west) {$2\quad\quad$};
\node (b2) at ( 6,0.6) [place] {};
\node [left] at (b2.west) {$5\quad\quad$};
\node (i2) at ( 6,0.2) [place] {};
\node (k2) [left] at (i2.west) {$7\quad\quad$};

\end{tikzpicture} 
\caption{The partition formed by the coroots in $\,R_s^+\,\check{}\,$ with respect to their special heights $\,n$ when $R$ is of type  $\,B_r\,$ (top left), $\,C_r\,$ (top right), $\,F_4\,$ (bottom left) and $\,G_2\,$ (bottom right).}\label{z12}
\end{figure}

\begin{remark}
By Theorem~\text{\ref{T1}} and \text{\cite[Table 1]{5Pan}}, when $\,R\,$ is of type $\,B, C\,$ or $\,F\,$ there exists a graded subspace $\,T=\langle T_1,\cdots,T_{r_s}\rangle \subset (\bigwedge V_{\theta_s})^{\mathfrak{g}},\,$ such that each $\,T_i\,$ is a primitive generator of degree $\,2h_i+1\,$ in $\,\bigwedge V_{\theta_s}\,$ and $\,(\bigwedge V_{\theta_s})^{\mathfrak{g}}\,$ is an exterior algebra over $\,T.\,$
 \end{remark}

\section{The graded multiplicity of \texorpdfstring{$\,V_{\theta_s}\,$}{Lg}}\label{p}
    In this section we use the action of the operator $\,Y^{\theta{}\,\check{}}\,$ from the double affine Hecke algebra $\,\mathcal{H}\!\!\!\mathcal{H}_{q,\,t_{s,\,l}}\,$ on a subset of the group algebra $\,\mathbb{Q}_{q,\,t_{s,\,l}}[\mathcal{P}],\,$ some properties of a subset of $\,R$ associated with $\,Y^{\theta{}\,\check{}}\,$ and the unitary property of $\,Y^{\theta{}\,\check{}}\,$ with respect to Cherednik's inner product on $\,\mathbb{Q}_{q,\,t_{s,\,l}}[\mathcal{P}]\,$ to find$\,GM_{\theta_s}(q)\,$ in $\,\bigwedge V_{\theta_s}.\,$ We set $\,t_{\alpha}=q^{\frac{-k(\alpha)}{2}}\,$ (see \text{\cite[Sect. 4]{17Mac}} and \text{\cite{3Baz}}) for $\,\alpha\in R.\,$
\par The following subsection is from \text{\cite{15Che,12Che}}, but we maintain $\,(k_s\,,k_l)=(2\,,1)\,$ for our unique integer label $\,k(\alpha)\,$ on the different $W$-orbits of $\,R. \,$

\subsection{Double affine Hecke algebra \texorpdfstring{$\,\mathcal{H}\!\!\!\mathcal{H}_{q,\,t_{s,\,l}}\,$}{Lg}}
Let $\,\hat{W}=\Pi \ltimes W^{a} \,$ be the extended affine Weyl group, where $\,W^{a}\,$ is the affine Weyl group and $\,\Pi\,$ is the subgroup of $\,\hat{W}\,$ which leaves the affine Dynkin diagram invariant. The group $\,\hat{W}\,$ is isomorphic to $\,W\ltimes \tau(\mathcal{P}\,\check{}\,),$ where the subgroup $\,\tau(\mathcal{P}\,\check{}\,)$ are translations of the Euclidean space $\,E\,$ associated with the root system $\,R\,$ of the Lie algebra $\,\mathfrak{g}\,$ by coweights. Let the double affine Hecke algebra $\,\mathcal{H}\!\!\!\mathcal{H}_{q,\,t_{s,\,l}}\,$ be the quotient of the field $\,\mathbb{Q}_{q,\,t_{s,\,l}}\,$ by elements $\,\{T_i\}_{i=0}^r,\,$ $\,\{e^{\omega_i}\}_{i=1}^r\,$ and the group $\,\Pi,\,$ modulo the relation $\,(T_i-t_i)(T_i+t_i^{-1})=0,\, 0\leq i \leq r,$ where $\,T_i,\, e^{\omega_i}\,$ and $\,\Pi\,$ generate and satisfy the relations on the double affine braid group $\,\mathscr{B}\,$ associated with $\,\mathcal{H}\!\!\!\mathcal{H}_{q,\,t_{s,\,l}}.\,$
\par Let elements $\,Y\,$ be pairwise commuting elements contained in the al\-gebra $\,\mathcal{H}\!\!\!\mathcal{H}_{q,\,t_{s,\,l}}\,$ such that \begin{equation}\label{z20}
    Y^{\lambda}=T(\tau(\lambda)),
\end{equation}
 where $\,\tau(\lambda)\in\hat{W}\,$ is the translation of $\,E\,$ by $\,\lambda\in \mathcal{P}\,\check{}.\,$ Then every element $\,H\in\mathcal{H}\!\!\!\mathcal{H}_{q,\,t_{s,\,l}}\,$ can be uniquely written as
\begin{equation*}
H=\sum_{w\in W} h_wT_wf_w,
\end{equation*}
where $\,h_w\in\mathbb{Q}_{q,\,t_{s,\,l}}[\mathcal{P}]\,$ and $\,f_w\in\mathbb{Q}_{q,\,t_{s,\,l}}[Y].\,$ Each $\,h_w\in\mathbb{Q}_{q,\,t_{s,\,l}}[\mathcal{P}]\,$ acts naturally on $\,\mathbb{Q}_{q,\,t_{s,\,l}}[\mathcal{P}]\,$ while $\,T_w=\prod_{i=1}^{p}T_{j_i}\,$ when $\,w\,$ has a reduced decomposition $s_{j_i}\cdots s_{j_p},\, 1\leq j_i\leq r,\,$ each $\,T_j\,$ acts on $\,\mathbb{Q}_{q,\,t_{s,\,l}}[\mathcal{P}]\,$ via the Demazure-Lusztig operator as
    \begin{equation*}
   T_j =t_{\alpha_j}s_j+(t_{\alpha_j}-t_{\alpha_j}^{-1})\frac{(1-s_{{\alpha_j}})}{1-e^{\alpha_j}},
\end{equation*}
and $\,Y\,$ acts on $\,\mathbb{Q}_{q,\,t_{s,\,l}}[\mathcal{P}]\,$ via \text{\eqref{z20}}.
\par Let $\,\theta\,$ be the highest root of $\,R.$ In what follows we consider the reduced decomposition of $\,\tau(\theta{}\,\check{})\in\hat{W}\,$ in order to determine the action of the operator  $\,Y^{\theta{}\,\check{}}\,$ on $\,\mathbb{Q}_{q,\,t_{s,\,l}}[\mathcal{P}].\,$

\subsection{Operator \texorpdfstring{$\,Y^{\theta\,\check{}}$}{Lg}}
Let $\,\hat{\alpha}=\alpha+n\delta\,$ be an affine root in the affine root system $\hat{R},\,$ where $\,\alpha\in R\,$ and $\,\delta\,$ is the constant function $\,1\,$ on $\,E.\,$ Let $\,G_{\hat{\alpha}}\,$ be the operator defined as
\begin{equation}\label{q}
   G_{\hat{\alpha}} =t_{\alpha}+(t_{\alpha}-t_{\alpha}^{-1})\frac{(s_{\hat{\alpha}}-1)}{1-q^{-n}e^{\alpha}},
\end{equation}
see \text{\cite[(2.17)]{15Che}} and \text{\cite[(2.3)]{3Baz}}.
Let $\,\hat{w}=\pi w \in \hat{W},\,$ where $\,\pi \in \Pi\,$ and $\,s_{j_p}\cdots s_{j_1},\,$ $0\leq j_i\leq r,\,$ is a reduced decomposition for $\,w\in W^{a}.\,$ Then by $\,(2.18)\,$ in \text{\cite{15Che}}
\begin{equation}\label{z19}
  T(\hat{w})=\hat{w}G_{\alpha^{(p)}}\cdots G_{\alpha^{(1)}},
\end{equation}
where $\,\alpha^{(1)},\cdots,\alpha^{(p)}\,$ is the chain of positive affine roots made negative by $\,\hat{w}\,$ and $\,\alpha^{(i)} = s_{j_{1}}\cdots s_{j_{i-1}}\alpha_{j_i}.\,$
\par The following is from \text{\cite[Sect. 3]{3Baz}}. We can write a reduced decomposition for $\,\tau(\theta{}\,\check{}\,)\in\hat{W}\,$ as
\begin{equation*}
    \tau(\theta{}\,\check{}\,)=s_{-\theta+\delta}s_{\theta}=s_0s_{j_p}\cdots s_{j_1}s_{j_0}s_{j_{-1}}\cdots s_{j_{-p}},\qquad 1\leq j_i \leq r,
\end{equation*}
such that $\,j_i=j_{-i},\, -p\leq i\leq p,\,$ where $\,\alpha_{0}=-\theta+\delta\,$ is the zeroth-simple root of $\,\hat{R}.\,$ Let $\,\hat{R}(\hat{w})\,$ be the set of positive affine root made negative by $\,\hat{w}.\,$ Then  $\,\hat{R}(\tau(\theta{}\,\check{}\,))=R(s_\theta) \cup \{\theta+\delta\},\,$
where $\,R(s_\theta)=\{\alpha \in R^+ \mid (\alpha\,\,,\,\theta{}\,\check{}\,)>0\}.\,$ Let
\begin{equation*}
    \alpha^{(-p)},\cdots,\alpha^{(0)},\cdots,\alpha^{(p)}, \alpha^{(p+1)}
\end{equation*}
be the chain of affine roots in $\,\hat{R}(\tau(\theta{}\,\check{}\,)),\,$ where $\,\alpha^{(i)} = s_{j_{-p}}\cdots s_{j_{i-1}}\alpha_{j_i},\, -p\leq i\leq p\,\text{ and }\, \alpha^{(p+1)}=\theta+\delta.\,$ Note that $\,
\alpha^{(-i)} = -s_{\theta}\alpha^{(i)}, \,\,  -p\leq i\leq p.\,$ This implies a symmetry in the lengths of roots about $\,\alpha^{(0)}\,$ in the chain of roots $\, \alpha^{(-p)},\cdots,\alpha^{(p)}\,$ in $\,\hat{R}(\tau(\theta{}\,\check{}\,)).\,$
By \text{\eqref{z20}} and \text{\eqref{z19}} therefore,
\begin{equation}\label{z21}
    Y^{\theta{}\,\check{}\,}=\tau(\theta{}\,\check{}\,)G_{\theta+\delta}G_{\alpha^{(p)}}\cdots G_{\alpha^{(-p)}}.
\end{equation}
\par We will use the action of $\,Y^{\theta{}\,\check{}}\,$ from \text{\eqref{z21}} on $\,e^0\,$ and $\,e^{\theta_s}\in \mathcal{Q}[q^{\pm 1},t_{s,\,l}^{\pm 1}][\mathcal{P}]\,$ to deduce $\,GM_{\theta_s}(q)\,$ in $\bigwedge V_{\theta_s}\,$ from \text{\eqref{a}}.

\subsubsection{The action of \texorpdfstring{$\,Y^{\theta{}\,\check{}}\,$}{Lg}}
\begin{proposition}\label{P1}
Let $\,R\,$ be of type $B, C \text{ or } F.\,$ The following holds for the action of $\,Y^{\theta{}\,\check{}}\,$ on $\,e^0\,$ and $\,e^{\theta_s}\,$ in terms of the double parameter $\,t_{s,\,l}.\,$
    \begin{align}\label{Ye}
    Y^{\theta{}\,\check{}} e^0=&\,t_l^{2L} t_s^S e^0,\nonumber\\
    Y^{\theta{}\,\check{}} e^{\theta_s}=&\, qt_l^{-2}t_s^{-S+2} e^{\theta_s}-(t_s-t_s^{-1})t_l^{L-1}t_s^{-S+3} e^0.
    \end{align}
where $\,L=ht_l\,\theta,\, S=ht_s\,\theta\, \text{ and }\, t_l \,\text{ (respectively } t_s) = t_{\alpha} \,$ when $\,\alpha \in R_l \,$ (respectively $\,R_s).\,$
\end{proposition}
\begin{remark}
The case of the double parameter $\,t_{s,\,l}\,$ for $\,t\,$ in the action of $\,Y^{\theta{}\,\check{}}\,$ in Proposition \text{\ref{P1}} provides a non-trivial and intricate extension of Bazlov's result in \text{\cite[Sect. 3]{3Baz}}. If we take $\,t_s=t_l=t,\,$ then \text{\eqref{Ye}} reduces to the weaker results obtained in $\,(11)\,$ and $\,(12)\,$ in \text{\cite[Sect. 3]{3Baz}}. 
\end{remark}
We will prove Proposition \text{\ref{P1}} using some properties of the roots in  $\,\hat{R}(\tau(\theta\,\check{}\,))\,$ to determine the actions of the operators $\,G_{\alpha}\,$ in \text{\eqref{z21}} on $\,e^{0}\,$ and $\,e^{\theta_s}.\,$

\subsubsection{Proof of Proposition \text{\ref{P1}}}
\begin{proof}
Following \text{\cite{3Baz}}, but using the double parameter $\,t_{s,\,l}\,$ instead of $t,\,$ we introduce another formula for $\,G_{\alpha}.\,$ Let
\begin{align}
h_{\alpha} &= t_{\alpha}-t_{\alpha}^{-1},\label{hi}\\
\varepsilon(\alpha\, , \, \beta) &=
\begin{cases}
-1, & (\alpha\, , \, \beta) > 0,\\
+1, & (\alpha\, , \, \beta) \leq 0,
\end{cases}
\qquad \alpha, \beta \in E.\nonumber
\end{align} 
We denote $\,\varepsilon(\alpha^{(i)}\, ,\,\alpha^{(j)\,\check{}}\,)\,$ by $\,\varepsilon_{i, \,j},\,\,\alpha^{(i)},\,\alpha^{(j)}\in R(s_{\theta}).\,$ Let $\, \alpha \in R,\, \mu \in \mathcal{P}.\,$ Then 
\begin{equation}\label{Gemu}
G_{\alpha} e^{\mu} = t_{\alpha}^{\varepsilon}\,e^{\mu} + \varepsilon\cdot h_{\alpha} \sum_{i=1}^{|(\mu\, , \, \alpha\,\check{}\,)|+ \frac{\varepsilon-1}{2}} e^{\mu+i\varepsilon \alpha}, \qquad \varepsilon = \varepsilon (\alpha\, , \, \beta).
\end{equation}
In particular
\begin{align}\label{Ge}
      G_{\alpha^{(k)}}\,e^{\alpha^{(i)}}&=t_{\alpha^{(k)}}^{\varepsilon_{k,i}}\,e^{\alpha^{(i)}}-\delta_{k,i\,}h_{\alpha^{(k)}}\,e^0, \qquad \alpha^{(i)} \in R_s(\tau(\theta\,\check{})),\nonumber\\
      G_{\alpha^{(k)}}\,e^0&=t_{\alpha^{(k)}}\,e^0.
\end{align}
See \text{\cite{3Baz}}. 
\par By Lemma~$1$ in \text{\cite{3Baz}} and \text{\eqref{Ge}} $\,G_{\theta+\delta}G_{\alpha^{(p)}}\cdots G_{\alpha^{(-p)}}e^0=t_l^{2L}t_s^{S}e^0.\,$ We apply $\,\tau(\theta{}\,\check{}\,),\,$ which is identity on $\,\mathcal{Q}[q^{\pm 1}, t_{s,\,l}^{\pm 1}]e^0,\,$ to this and obtain $\,Y^{\theta{}\,\check{}} e^0\,$ in \text{\eqref{Ye}}.
To calculate $\,Y^{\theta{}\,\check{}\,}e^{\theta_s}\,$ we first compute $\,G_{\alpha^{(p)}}\cdots G_{\alpha^{(-p)}}e^{\theta_s}\,$ and then apply $\,\tau(\theta{}\,\check{}\,)G_{\theta+\delta}\,$ to the result. By Lemma~$4\,(f)\,$ in \text{\cite{3Baz}} there exists an index $i\,$ such that $\,\theta_s = \alpha^{(i)} \in R(s_{\theta}).\,$ We fix this $\,i.\,$ Using \text{\eqref{Ge}} we obtain 
\begin{equation}\label{Rf}
    G_{\alpha^{(p)}}\cdots G_{\alpha^{(-p)}}e^{\alpha^{(i)}}=\prod_{k=-p}^p t_{\alpha^{(k)}}^{\varepsilon_{i, \,k}}\,e^{\alpha^{(i)}}-h_s\prod_{k=i+1}^p t_{\alpha^{(k)}}\prod_{k=-p}^{i-1} t_{\alpha^{(k)}}^{\varepsilon_{i, \,k}}\,e^0.
\end{equation}
We will simplify each of $\,\prod_{k=-p}^p t_{\alpha^{(k)}}^{\varepsilon_{i, \,k}},\,$ $\,\prod_{k=i+1}^p t_{\alpha^{(k)}}\,$ and $\,\prod_{k=-p}^{i-1} t_{\alpha^{(k)}}^{\varepsilon_{i, \,k}}\,$ to obtain $\,G_{\alpha^{(p)}}\cdots G_{\alpha^{(-p)}}e^{\theta_s},\,$ in terms of the double parameter $\,t_{s,\,l},\,$ in $\,\mathbb{Q}_{q,\,t_{s,\,l}}[\mathcal{P}].\,$ By Lemma~$4\,(b),(c)\,$ in \text{\cite{3Baz}} we have
\begin{align}\label{ei}
    \varepsilon_{i, \,k}&=-1 \qquad \text{ \,\,when }\,\, \alpha^{(i)},\, \alpha^{(k)}\in R_s,\, k\neq -i,\nonumber \\
    \varepsilon_{i,-i}&=1 \qquad \quad \,\text{ when }\,\, \alpha^{(i)}\in R_s,\nonumber\\ 
    t_{\alpha^{(k)}}^{\varepsilon_{i, \,k}}\,t_{\alpha^{(k)}}^{\varepsilon_{i,-k}}&=1 \qquad \quad \text{ when }\,\, \alpha^{(i)}\in R_s,\,\alpha^{(k)}\in R_l,\, k\neq 0.
\end{align}
Let $\,\alpha \in R,\,$ then \begin{equation}\label{Lemma2}
    (\theta_s\,\,,\alpha\,\check{}\,)\in \{0,\pm 1\},
\end{equation}
see \text{\cite[Chap.VI, Sect. 1.3]{14Bou}}. Hence, using Lemma~$4\,(a)\,$  and Lemma~$1\,$ in \text{\cite{3Baz}}
\begin{equation}\label{Re}
    \prod_{k=-p}^p t_{\alpha^{(k)}}^{\varepsilon_{i, \,k}}= t_{\alpha^{(-i)}}\,t_{\alpha^{(0)}}^{\varepsilon_{i,0}}\,\times\prod_{\alpha^{(k)}\in R_s \backslash \{\alpha^{(i)}\}} t_{\alpha^{(k)}}^{-1}=t_s^{-S+2\,}t_l^{-1}.
\end{equation}
Since $\, \theta_s\,$ is the highest short root in the sequence of short roots in the chain $\, \alpha^{(-p)},\cdots,\alpha^{(p)}\,$ (Lemma~$4\,(e)\,$ in \text{\cite{3Baz}}), it follows that the set of roots $\,\{\alpha^{(j)}\, |\, i<j\}\,$ in the chain $\, \alpha^{(-p)},\cdots,\alpha^{(p)},\,$ where $\, \theta_s=\alpha^{(i)},\,$ are all long roots. Therefore
\begin{equation}\label{Rd}
    \prod_{k=i+1}^p t_{\alpha^{(k)}}=t_l^{p-i}.
\end{equation}
By the symmetry in the lengths of roots about $\, \alpha^{(0)}\,$ and since all the short roots in $\,R(s_{\theta})\,$ lie between $\,\alpha^{(-i)}\,$ and $\,\alpha^{(i)},\,$ combining \text{\eqref{Lemma2}} and \text{\eqref{ei}} gives 
\begin{align}
    \prod_{k=-p}^{i-1} t_{\alpha^{(k)}}^{\varepsilon_{i, \,k}}&=t_s^{\varepsilon_{i,-i}}\,t_l^{\varepsilon_{i,0}}\prod_{k=-p}^{-i-1} t_l^{\varepsilon_{i, \,k}}\,\times \prod_{\alpha^{(k)} \in R_s^+(\theta)\backslash\{\alpha^{(\pm i)}\}} t_s^{\varepsilon_{i, \,k}} \,\times \prod_{\substack{k=-i+1\\k\neq 0}}^{i-1} t_l^{\varepsilon_{i, \,k}}\nonumber\\
    &=t_l^{-1}t_s^{-S+3}\prod_{k=-p}^{-i-1} t_l^{\varepsilon_{i, \,k}}.\label{Ra}
\end{align}
From (21) in \text{\cite{3Baz}} and \text{\eqref{Ra}} above
\begin{align*}
    \Big(\sum_{k=-p}^{i-1} \varepsilon_{i, \,k}\Big)&=\Big(\sum_{k=-p}^{-i-1} \varepsilon_{i, \,k}\Big)-S+2\\
    &=p+i+2-2ht\,\alpha^{(i)}.
\end{align*}
By Lemma~$4\,(f)\,$  in \text{\cite{3Baz}} therefore,
\begin{equation*}
    \sum_{k=-p}^{-i-1} \varepsilon_{i, \,k}=p+i-S-L+1.
\end{equation*}
Hence,
\begin{equation}\label{Rc}
    \prod_{k=-p}^{i-1} t_{\alpha^{(k)}}^{\varepsilon_{i, \,k}}=t_l^{p+i-L-S}t_s^{-S+3}.
\end{equation}
By Lemma~$1$ in \text{\cite{3Baz}} $\,R(s_{\theta})=2p+1=2L+S+1.\,$
Combining \text{\eqref{Rf}}, \text{\eqref{Re}}, \text{\eqref{Rd}} and \text{\eqref{Rc}} therefore gives
\begin{equation}\label{Rg}
    G_{\alpha^{(p)}}\cdots G_{\alpha^{(-p)}}e^{\theta_s}=t_l^{-1}t_s^{-S+2}\,e^{\theta_s}-h_s\,t_l^{L-2}t_s^{-S+3}\,e^0.
\end{equation}
Let $\,w\tau(\lambda)\in\hat{W},\,$ where $\,w\in W\,$ and $\,\lambda \in \mathcal{P}\,\check{}.\,$ Then
\begin{equation}\label{Rj}
    w\tau(\lambda)(e^y)=q^{(\lambda\,\,,y)}e^{w(y)},
\end{equation}
see \text{\cite[Sect. 2.3]{3Baz}}. By \text{\eqref{q}}, \text{\eqref{hi}} and \text{\eqref{Rj}} $\,\tau(\theta{}\,\check{}\,) \,G_{\theta+\delta}(e^{\theta_s})=qt_l^{-1}\,e^{\theta_s}\,$ and 
$\,\tau(\theta{}\,\check{}\,)\cdot G_{\theta+\delta}(e^0)=t_le^0.\,$ Applying $\,\tau(\theta{}\,\check{}\,)G_{\theta+\delta}\,$ to \text{\eqref{Rg}} then gives $\,Y^{\theta{}\,\check{}}\,e^{\theta_s}\,$ in \text{\eqref{Ye}}. 
\end{proof}

\subsubsection{The action of \texorpdfstring{$\,Y^{\theta{}\,\check{}}\,$}{Lg} on the root system of \texorpdfstring{$\,G_2\,$}{Lg}}

Let $R\,$ be of type $\,G_2\,$ with short and long simple root $\,\alpha_1\,$ and $\alpha_2\,$ respectively. We denote by $\,\beta\,$ and $\,\theta_s\,$ the short roots $\,\alpha_1 + \alpha_2\,$ and $\,2\alpha_1 + \alpha_2 \,$ respectively, by $\,\gamma\,$ and  $\,\theta\,$ the long roots $\,3\alpha_1 + \alpha_2\,$ and  $\,3\alpha_1 + 2\alpha_2 \,$ respectively. The reflection $\,s_{\theta}\,$ has the reduced decomposition $\,s_2s_1s_2s_1s_2\,$ and
\begin{equation}\label{z17}
    Y^{\theta{}\,\check{}}=\tau(\theta{}\,\check{}\,)G_{\theta+\delta\,}G_{\gamma\,}G_{\theta_s\,}G_{\theta\,}G_{\beta\,}G_{\alpha_2}.
\end{equation}
See \text{\cite[Sect. 3.8]{3Baz}}.
\begin{proposition}
 The action of $\,Y^{\theta{}\,\check{}}\,$ on $\,e^0\,$ and $\,e^{\theta_s}\,$ in terms of the double parameter $t_{s,\,l}\,$ is given as
\begin{align}\label{GYe}
    Y^{\theta{}\,\check{}} e^0=&\,t_l^{4} t_s^2 e^0,\nonumber\\
    Y^{\theta{}\,\check{}} e^{\theta_s}=&\,qt_l^{-2}t_s^{-2} e^{\theta_s}-(t_s-t_s^{-1})t_l^2t_s^{-1} e^0.
    \end{align}
\end{proposition}

\begin{proof}
Apply $\,G_\alpha\,$ and $\,\tau(\theta{}\,\check{}\,)G_{\theta+\delta}\,$ from \text{\eqref{z17}} to $\,e^0\,$ and $\,e^{\theta_s}\,$ using \text{\eqref{q}},  \text{\eqref{Ge}} and \text{\eqref{Rj}}.
\end{proof}

We will use $\,Y^{\theta{}\,\check{}} e^{\theta_s},\,$ $\,Y^{\theta{}\,\check{}} e^{0}\,$ and the unitary property of $\, Y^{\theta{}\,\check{}}\,$ with respect to Cherednik's inner product to find $\,(e^{\alpha}\,,\,1)\,$ for all $\,\alpha\,$ in $\,R_s.\,$ We will then use these to calculate $\,\big(\chi(V_{\theta_s})\,,\,1\big)\,$ from which we will deduce $\,GM_{\theta_s}(q)\,$ in $\,\bigwedge V_{\theta_s}.\,$
\par The next two lemmas deal respectively with the linear combination of $\,\alpha\,$ and $\,(e^\alpha\,\,,1),\, \alpha \in \,R_s^+,\,$ with respect to the dominant root $\,\theta_s\,$ of $\,R_s.$
\begin{lemma}\label{r}
Let $R\,$ be of type $\,B, C, F\,$ or $\,G\,$ and let $\,\theta_s\,$ be the dominant root of $\,R_s.\,$ If $\,\alpha \in R_s^+\backslash \{\theta_s\},\,$ then there exists $\,s_{i_1}\cdots s_{i_k}\in W\,$ such that
\begin{equation*}
    \alpha=s_{i_1}\cdots s_{i_k}\theta_s=\theta_s-(\alpha_{i_1}+\cdots+ \alpha_{i_k}),
\end{equation*}
where $\,s_{i_j}\,$ is the simple reflection along $\,\alpha_{i_j}\in \Phi.\,$
\end{lemma}
\begin{proof}
Recall the following. (1) The set of dominant weights $\, \mathcal{P}^+\,$ of $\,\mathfrak{g}\,$ is $\,\{\omega \in \mathcal{P} \mid (\omega\,\,,\alpha_i\check{}\,)\geq 0\,\, \forall \,\, \alpha_i\in\Phi\},\,$ \text{\cite[Chap.VI, Sect. 1.10]{14Bou}}. (2) $\mathcal{P}^+\cap R_s=\{\theta_s\},\,$ \text{\cite[Sect. 2]{16Ste}}. (3) $\,(\alpha\,\,,\,\beta\,\check{}\,)\in \{0, \pm 1\}\,$ for $\,\alpha \in R_s,\, \beta \in R,\,$ \text{\cite[Chap.VI, Sect. 1.3]{14Bou}}. Hence, if $\,\alpha \in R_s^+\backslash \{\theta_s\},\,$ then there exists $\,\alpha_i \in \Phi \,$ such that $\,s_i(\alpha)=\alpha+\alpha_i.\,$ By Weyl conjugacy there exists $\,s_{i_k}\cdots s_{i_1}\in W\,$ such that
\begin{equation*}
    \theta_s=s_{i_k}\cdots s_{i_1}\alpha=\alpha+\alpha_{i_1}+\cdots+ \alpha_{i_k}.
\end{equation*}
The lemma holds since each $\,s_{i_j}\,$ is an involution.
\end{proof}

\begin{lemma}\label{Amax}
Let $R\,$ be of type $\,B, C, F\,$ or $\,G.\,$ If there exists a constant $\,X \in \mathbb{Q}[q^{\pm 1}, t_{s,\,l}^{\pm 1}]\,$ such that the formula
\begin{equation}\label{Xht}
    (e^{\lambda}\,,\,1)=t_l^{-2ht_l\,\lambda}t_s^{-2ht_s\,\lambda}X
\end{equation}
holds when $\lambda$ is the dominant weight $\,\theta_s\,$ in $\,R_s,\,$ then the formula holds when $\lambda=\alpha\,$ for all $\,\alpha \in R^+_s.\,$
\end{lemma}
\begin{proof}
Let $\,\beta, \alpha \in R^+_s\,$ such that 
\begin{equation*}
    \alpha=s_{i_j}\beta=\beta-\alpha_{i_j}.
\end{equation*}
By \text{\eqref{z19}} and \text{\eqref{Gemu}}
\begin{equation*}
    T_{i_j}e^{\beta}=s_{i_j}G_{\alpha_{i_j}}e^{\beta}=t_{\alpha_{i_j}}^{-1}e^{\alpha}.
\end{equation*}
Since $\,T\,$ is unitary with respect to Cherednik's inner product,
\begin{equation*}
    t_{\alpha_{i_j}}^{-1}(e^{\alpha}\,,\,1)=t_{\alpha_{i_j}}(e^{\beta}\,,\,1).
\end{equation*}
Therefore, if \text{\eqref{Xht}} holds for $\,\beta,\,$ then
\begin{equation*}
    (e^{\alpha}\,,\,1)=t_l^{-2ht_l\,\alpha}t_s^{-2ht_s\,\alpha}X.
\end{equation*}
Apply Lemma~\text{\ref{r}} and the lemma is proved.
\end{proof}
\subsection{Calculating \texorpdfstring{$\,GM_{\theta_s}(q)$}{Lg}}
Let $R\,$ be of type $\,B, C, F\,$ or $\,G.\,$  Recall the unitary property of $\,Y^{\theta{}\,\check{}}\,$ with respect to Cherednik's inner product $\,(\,\,,\,)\,$ on the algebra $\,\mathbb{Q}_{q,\,t_{s,\,l}}[\mathcal{P}]\,$ implies 
\begin{equation*}
    (Y^{\theta{}\,\check{}}\,e^{\theta_s}\,,\,1)=(e^{\theta_s}\,,\,Y^{-\theta{}\,\check{}}1).
    \end{equation*}
Using $\,Y^{\theta{}\,\check{}}\,e^{\theta_s}\,$ and $\,Y^{\theta{}\,\check{}}\,e^0\,$ in \text{\eqref{Ye}} and \text{\eqref{GYe}} we have
\begin{equation*}
    (e^{\theta_s}\,,\,1)=\frac{t_l^{-2ht_l\,\theta_s}t_s^{-2ht_s\,\theta_s}}{qt_l^{-2L-2}t_s^{-2S+2}-1}(t_s^2-1)(1\,,\,1),
\end{equation*}
where $\,L=ht_l\,\theta, \,S=ht_s\,\theta,\, ht_l\,\theta_s=\frac{L+1}{2} \,$ and  $\, ht_s \,\theta_s=S-1\,$ (Lemma~$4\,(d),$ $(e)\,$ in \text{\cite{3Baz}}).
Let $\,\alpha\,$ in $\,R_s^+.\,$ By Lemma~\text{\ref{Amax}}
\begin{equation}\label{s}
    (e^{\alpha}\,,\,1)=\frac{t_l^{-2ht_l\,\alpha}t_s^{-2ht_s\,\alpha}}{qt_l^{-2L-2}t_s^{-2S+2}-1}(t_s^2-1)(1\,,\,1),\qquad \,\text { for all } \alpha \in R_s^+.
\end{equation}

By \text{\cite[Sect. 2.5]{3Baz}} and \text{\eqref{s}}
\begin{align*}
    (e^{-\alpha}\,,\,1)&=(e^{\alpha}\,,\,1)^{\bar{*}}\nonumber\\
    &=\frac{qt_l^{2(ht_{l}\,\alpha-(L+1))}t_s^{2(ht_{s}\,\alpha-S)}}{qt_l^{-2L-2}t_s^{-2S+2}-1}(t_s^{2}-1)(1\,,\,1), \qquad \,\text { for all } \alpha \in R_s^+,
\end{align*}
where $\,\bar{*}\,$ is the involution on $\,\mathbb{Q}_{q,\,t_{s,\,l}}[\mathcal{P}]\,$ which acts on $\,q, t_{s,\,l}, \text{ and } e^{\alpha}\,$ as follows.
\begin{equation*}
    \bar{*}:t_{s,\,l}\mapsto t_{s,\,l}^{-1}\,, \qquad q\mapsto q^{-1},\qquad e^{\alpha}\mapsto e^{\alpha}.
\end{equation*}
Hence,
\begin{align}\label{v}
    \frac{(\sum_{\alpha \in R_s}\,e^{\alpha}\,,\,1)}{(1\,,\,1)}&=\frac{t_s^{2}-1}{qt_l^{-2(L+1)}t_s^{-2(S-1)}-1}\nonumber\\
    &\times \sum_{\alpha \in R_s^+} (t_l^{-2ht_{l}\,\alpha}t_{s}^{-2ht_{s}\,\alpha}+qt_l^{2ht_{l}\,\alpha-2(L+1)}t_{s}^{2ht_{s}\,\alpha-2S}).
\end{align}
Recall the character of the little adjoint module 
$\,
    \chi_{\theta_s} = r_s+\sum_{\alpha\in R_s}\,e^{\alpha},
\,$
where $\,r_s=\# (R_s \cap \Phi).\,$ Therefore,
\begin{equation}\label{z}
    \frac{\big(\chi_{\theta_s}\,,\,1\big)}{(1\,,\,1)}=r_s+\frac{(\sum_{\alpha \in R_s}\,e^{\alpha}\,,\,1)}{(1\,,\,1)}.
\end{equation}
Let $\,\lambda \in \mathcal{P}^+,\,$ then $\,\frac{\langle\chi_{\lambda}\,\,,\,1\rangle}{\langle 1\,\,,\,1\rangle}=\frac{(\chi_{\lambda}\,\,,\,1)}{(1\,\,,\,1)}\,$ since $\,\chi_{\lambda}\,$ and $1$ are $\,W-$invariant, see \text{\cite[(5.1.38)]{10Mac}}.
By \text{\eqref{a}} then,
\begin{equation*}
  GM_{\lambda}(-q)=GM_{0}(-q)\frac{(\chi_{\lambda}\,,\,1)_{2,1}}{(1\,,\,1)_{2,1}}.
\end{equation*}
We set
$\,
    t_{\alpha}=q^{\frac{-k(\alpha)}{2}},  
\,$
see \text{\cite[Sect. 4]{17Mac}} and \text{\cite{3Baz}}. Recall from Sect.\text{\eqref{zz}} that $\,(\sigma_k \, , \, \alpha)=ht_l\,\alpha+2ht_s\,\alpha,\,$ where $\,ht_m\,\alpha=\sum_{\alpha_i\in\Phi_m}\lambda_i,\,\,m\in\{s,\,l\},\,$ when $\,\alpha=\sum_{i=1}^r \lambda_i\alpha_i.\,$ Using \text{\eqref{v}}, \text{\eqref{z}} and $\,(k_s,k_l)=(2,1), \,GM_{\theta_s}(-q)\,$ becomes
\begin{equation}\label{t}
    GM_{\theta_s}(-q)=GM_{0}(-q)\bigg(r_s+\big(\frac{1-q^{-2}}{1-q^{L+2S}}\big)\sum_{\alpha \in R_s^+} \big(q^{(\sigma_k \, , \, \alpha)}+q^{2S+L+2-(\sigma_k \, , \, \alpha)}\big)\bigg).
\end{equation}
We will simplify this case by case for each $\,R\,$ of type $\,B, C, F\,$ or $\,G.\,$

\subsubsection{\texorpdfstring{Type $\,B_r\,$}{Lg}}
Recall the following in the root system of type $\,B_r.\,$
\begin{align*}
    R_s^+=&\{\sum_{i\leq m\leq r} \alpha_m, \,\, 1\leq i \leq r\},\quad \Phi_l=\{\alpha_1,\cdots,\alpha_{r-1}\},\quad \Phi_s=\{\alpha_r\},\\
    \theta=&\alpha_1+2\alpha_2 + 2\alpha_3+\cdots+2\alpha_r.
\end{align*}See \text{\cite[Plate II]{14Bou}}. If $\,\alpha=\alpha_i+\cdots+\alpha_r,\,$ then $\,(\sigma_k \, , \, \alpha)=r+2-i.\,$ Using $\,GM_{0}(q)\,$ in \text{\eqref{j}}, $\,GM_{\theta_s}(q)\,$ in \text{\eqref{t}} simplifies to
\begin{equation}\label{z27}
    GM_{\theta_s}(q)=q+q^{2r}.
\end{equation}
This result in \text{\eqref{z27}} implies that when $\,\mathfrak{g}\,$ is of type $B_r$ the only copies of the $2r+1$-dimensional little adjoint module in its exterior algebra are the exterior powers $\bigwedge^1 V_{\theta_s}\,$ and $\bigwedge^{2r} V_{\theta_s},\,$ which are known to be isomorphic to the little adjoint module as $\mathfrak{g}$-modules (recall the basis of $\bigwedge^1 V_{\theta_s}\,$ and its Poincar\'e duality to $\bigwedge^{2r} V_{\theta_s},\,$ \text{\cite{1Me, 8Wil}}).  
\subsubsection{\texorpdfstring{Type $\,C_r\,$}{Lg}}
Recall that in the root system of type $\,C_r\,$ \begin{align*}
    R_s^+&=J \cup K=\{\sum_{i\leq m<j} \alpha_m \} \, \bigcup \, \{\sum_{i\leq m < j} \alpha_m+\sum_{j\leq m < r}2\alpha_m + \alpha_r\}, \\
    & 1\leq i<j \leq r.\nonumber \\
    (R_s^+)\,\check{}&=J\,\check{} \cup K\,\check{}, \quad\text{ where }\, J\,\check{}=\{\sum_{i\leq m < j} \alpha\,\check{}_m, \quad 1\leq i<j \leq r\},\\
    K\,\check{}&=\{\sum_{i\leq m < j} \alpha\,\check{}_m+ \sum_{j \leq m \leq r}2\alpha\,\check{}_m, \quad 1\leq i<j \leq r \}.\\
    \Phi_s&=\{\alpha_1,\cdots,\alpha_{r-1}\},\,\Phi_l=\{\alpha_r\},\,  \theta=2\alpha_1+2\alpha_2+\cdots +2\alpha_{r-1} +2\alpha_{r}.
\end{align*}
See \text{\cite[Plate III]{14Bou}}. Let $\,Y\subseteq R_s^+$ and let $\,h_{Y}(n)=\#\{\alpha\in Y \mid (\sigma_k \, , \, \alpha)=n\}. \,$ By \text{\eqref{b}} then
\begin{equation*}
   (\sigma_k \, , \, \alpha)=(\rho_k\,\,,\alpha\,\check{}\,) \qquad \text{for all }\,\alpha \in J.
\end{equation*}
Hence,
\begin{equation}\label{z13}
    h_{J}(n)=h_{J\,\check{}\,}(n).
\end{equation}
Note that $\,ht_l\,\alpha=0\,$ for all $\,\alpha \in J.\,$ In $\,K,$
\begin{equation*}
    (\sigma_k \, , \, \alpha)=(\rho_k\,\,,\alpha\,\check{}\,)-1 \qquad \text{for all }\,\alpha \in K.
\end{equation*}
Therefore,
\begin{equation}\label{z10}
h_{K}(n-1)=h_{K\,\check{}\,}(n).
\end{equation}
Table \text{\ref{(s_ks,a)CrK}} below is from \text{\eqref{z10}} and Table \text{\ref{(r_k,a)CrZ}}.
\begin{table}[htb]
\centering
\begin{tabular}{ |c|c|c| }
\hline
$n$ & $h^o_{K}(n)$ & $h^e_{{K}}(n)$\\
\hline
$3$ & $1$ & $1$\\
\hline
$5$  & $1$ & $1$\\
\hline
$7$ & $2$ & $2$\\
\hline
$9$ & $2$ & $2$\\
\hline
$\vdots$ & $\vdots$ & $\vdots$\\
\hline
$2r- 5$ & $\frac{r-3}{2}$ & $\frac{r-2}{2}$\\
\hline
$2r- 3$ & $\frac{r-1}{2}$ & $\frac{r-2}{2}$\\
\hline
$2r-1$ &$\frac{r-1}{2}$ & $\frac{r}{2}$\\
\hline
$2r+1$ & $\frac{r-1}{2}$& $\frac{r-2}{2}$\\
\hline
$2r+3$ &$\frac{r-3}{2}$ & $\frac{r-2}{2}$\\
\hline
$\vdots$ & $\vdots$ & $\vdots$ \\
\hline
$4r-7$ & $1$ & $1$\\
\hline
$4r-5$ & $1$ & $1$\\
\hline
\end{tabular}
\caption{The multiplicity of the special heights of the  roots of the subset $K$ of $R_s^+\,$ in type $\,C_r$}\label{(s_ks,a)CrK}
\end{table}

Observe that when $\,r\,$ is odd
\begin{align}\label{x}
     &h^o_{{K}}(4p-1)=h^o_{{K}}(4p+1)=h^o_{{K}}(4r-4p-1)=h^o_{{K}}(4r-4p-3)=p,\nonumber\\
    & p=1,\cdots, \frac{r-1}{2}.
\end{align}
Whereas, when $\,r\,$ is even 
\begin{align}\label{y}
     &h^e_{{K}}(4p-1)=h^e_{{K}}(4p+1)=h^e_{{K}}(4r-4p-1)=h^e_{{K}}(4r-4p-3)=p,\nonumber\\
     &p=1,\cdots, \frac{r-2}{2};\nonumber\\
     &h^e_{{K}}(2r-1)=\frac{r}{2}.
      \end{align}
Therefore, using \text{\eqref{m}},\,\text{\eqref{z13}},\,\text{\eqref{z10}}\,\text{\eqref{x}},\,\text{\eqref{y}}, $\,h_{J\,\check{}\,}(n)\,$ in Table \text{\ref{(r_k,a)CrZ}} \,and $\,h_{K}(n)\,$ in Table \text{\ref{(s_ks,a)CrK}},
\begin{equation*}
GM_{\theta_s}(-q)=GM_{0}(-q)\bigg(r_s+\big(\frac{1-q^{-2}}{1-q^{L+2S}}\big)\sum_{\alpha \in J \cup K} \big(q^{(\sigma_k \, , \, \alpha)}+q^{2S+L+2-(\sigma_k \, , \, \alpha)}\big)\bigg)
\end{equation*}
when $\,r\,$ is odd becomes
\begin{align*}
        GM_{\theta_s}(-q)&=\prod_{p=1}^{r-2} (1-q^{4p+1})\Bigg((r-1)(1-q^{4r-3})+(1-q^{-2})\bigg( \sum_{p=1}^{r-1}\big((r-p)\times\nonumber\\
        &(q^{2p}+q^{4r-1-2p})\big)+\big(\sum_{p=1}^{\frac{r-3}{2}} (q^{4p+1}+q^{4r-1-(4p+1)}+q^{4p-1}+\nonumber\\
        &q^{4r-1-(4p-1)}+q^{4r-4p-1}+q^{4r-1-(4r-4p-1)}+q^{4r-4p-3}+\nonumber\\
        &q^{4r-1-(4r-4p-3)})\big)+
        \frac{r-1}{2}\big(q^{2r+1}+q^{4r-1-(2r+1)}+q^{2r-1}+\nonumber\\
        &q^{4r-1-(2r-1)}+q^{2r-3}+q^{4r-1-(2r-3)}\big)\bigg)\Bigg),
\end{align*}
and when $\,r\,$ is even $\,GM_{\theta_s}(-q)\,$ becomes
\begin{align*}
        GM_{\theta_s}(-q)&=\prod_{p=1}^{r-2} (1-q^{4p+1})\Bigg((r-1)(1-q^{4r-3})+(1-q^{-2})\bigg( \sum_{p=1}^{r-1}\big((r-p)\times\nonumber\\
        &(q^{2p}+q^{4r-1-2p})\big)+\big(\sum_{p=1}^{\frac{r-2}{2}} (q^{4p+1}+q^{4r-1-(4p+1)}+q^{4p-1}+\nonumber\\
        &q^{4r-1-(4p-1)}+q^{4r-4p-1}+q^{4r-1-(4r-4p-1)}+\nonumber\\
        &q^{4r-4p-3}+q^{4r-1-(4r-4p-3)})\big)+\frac{r}{2}\big(q^{2r-1}+q^{4r-1-(2r-1)}\big)\bigg)\Bigg).
\end{align*}
In both cases $\,GM_{\theta_s}(-q)\,$ simplifies such that
\begin{equation*}
   GM_{\theta_s}(q)=\bigg(\prod_{p=1}^{r-2} (1+q^{4p+1})\bigg) \sum_{p=1}^{r-1} (q^{4p-3}+q^{4p}).
\end{equation*}

\subsubsection{\texorpdfstring{Type $\,F_4$}{Lg}}
Recall that in the root system of $\,F_4,\,$ $\,\theta=2\alpha_1+3\alpha_2 + 4\alpha_3+2\alpha_4\,$ when $\,\Phi_l=\{\alpha_1, \alpha_2\}\,$ and $\,\Phi_s=\{\alpha_3,\alpha_4\}.\,$ See \text{\cite[Plate II]{14Bou}}. We use $\,GM_{0}(q)\,$ in \text{\eqref{u}} and $\,(\sigma_k \, , \, \alpha), \,\alpha \in R_s^+,\,$ in Table \text{\ref{n}} to simplify $\,GM_{\theta_s}(q)\,$ in \text{\eqref{t}} and obtain
\begin{equation*}
    GM_{\theta_s}(q)=(1+q^9)(q+q^8+q^9+q^{16}).
\end{equation*}

\subsubsection{\texorpdfstring{Type $\,G_2\,$}{Lg}}
Recall that in $\,G_2,\,$ if $\,\Phi_l=\{\alpha_{1}\}\,$ and $\,\Phi_s=\{\alpha_2\},\,$ then
\begin{equation*}
    R_s^+=\{\alpha_1, \alpha_1+\alpha_2, 2\alpha_1+\alpha_2\},\quad \text{and }\,\theta=3\alpha_1+2\alpha_2.
\end{equation*}
Using \text{\eqref{w}} $\,GM_{\theta_s}(q)\,$ in \text{\eqref{t}} simplifies to
\begin{equation}\label{z29}
    GM_{\theta_s}(q)=(1+q^3)(q+q^2+q^3).
\end{equation}

\begin{remark}
Given the small sizes of $\,R_s\,$ in type $\,B_2\,$ and $\,G_2,\,$ the reader can verify $\,GM_0(q)\,$ and $\,GM_{\theta_s}(q)\,$ in \text{\eqref{j}}, \text{\eqref{w}}, \text{\eqref{z27}} and \text{\eqref{z29}} by decomposing the exterior powers of $\,\bigwedge V_{\theta_s}\,$ in both cases of $\,R\,$ into their irreducible modules from their highest weight vectors. 
\end{remark}
\begin{theorem}\label{T2}
Let $R\,$ be of type $\,B, C\,$ or $\,F.\,$ The graded multiplicity of the little adjoint module $\,V_{\theta_s}\,$ in its exterior algebra $\,\bigwedge V_{\theta_s}\,$ is given by
 \begin{equation}\label{z9}
    GM_{\theta_s}(q)=\prod_{i=1}^{r_s-1}(1+q^{2h_i+1})\sum_{i=1}^{r_s}(q^{2h_i-(2h_1-1)}+q^{2h_i});
\end{equation}
while when $\,R\,$ is of type $\,G_2\,$
\begin{equation*}
    GM_{\theta_s}(q)=(1+q^{h_1})(q+q^{2}+q^3),
\end{equation*}
where $\, h_1,\cdots,h_{r_s}\,$ are the special exponents of $\,(R_s^+)\,\check{}.\,$
\end{theorem}
\begin{remark}
Let $R\,$ be of type $\,B, C\,$ or $\,F.\,$ Observe the following. 
\begin{enumerate}[1.]
    \item The dimension of the isotypic component of $\,V_{\theta_s}\,$ in $\,\bigwedge V_{\theta_s},\,$ $GM_{\theta_s}(1)=2^{r_s}r_s.\,$
We note that in the case of the multiplicity of the adjoint module $\,\mathfrak{g}\,$ in its exterior algebra, Kostant \text{\cite{1Kos}} proved $\,GM_{[\bigwedge \mathfrak{g}:\mathfrak{g}]}(1)=2^{r}r.\,$

\item The formula for $\,GM_{[\bigwedge V_{\theta_s}:V_{\theta_s}]}(q)\,$ in \text{\eqref{z9}} suggests that the isotypic component of $\,V_{\theta_s}\,$ is a free module over a subring of the invariants in $\,\bigwedge V_{\theta_s}.\,$ Recall that in the exterior algebra of $\,\mathfrak{g}\,$ 
\begin{equation*}
GM_{[\bigwedge \mathfrak{g} : \mathfrak{g}]}(q)=\prod_{i=1}^{r-1}(1+q^{2m_i+1})\sum_{i=1}^{r}(q^{2m_i-1}+q^{2m_i}),
\end{equation*}
where $\,m_i, 1\leq i\leq r,\,$ are the exponents of $\,R\,$  and $\,(\bigwedge \mathfrak{g})^{\mathfrak{g}}=\bigwedge(P_1,\cdots,$ $P_r),\,$ with each primitive invariant generator $\,P_i\,$ of degree $\,2m_i+1\,$ in $\,\bigwedge\mathfrak{g},\,$ \text{\cite{3Baz}}. DeConcini, Papi and Procesi proved that the isotypic component of $\,\mathfrak{g}\,$ is a free module of rank $\,2r\,$ over $\,\bigwedge(P_1,\cdots,P_{r-1}),\,$ with basis vectors $\,f_i\,$ and $\,u_i,\,$ of degree $\,2m_i\,$ and $\,2m_i-1\,$ respectively in $\,\bigwedge \mathfrak{g},\,$ in terms of derivations of $\,\bigwedge \mathfrak{g}\,$ on primitive invariants. (See \text{\cite{4DPP}}). Motivated by this result of De Concini, Papi and Procesi in $\,\bigwedge \mathfrak{g},\,$ we conclude this paper with a conjecture on the isotypic component of $\,V_{\theta_s}\,$ in $\,\bigwedge V_{\theta_s}.\,$ 
\end{enumerate}
\end{remark}

\begin{conjecture}
 Let $\,\bigwedge V_{\theta_s}\,$ be the exterior algebra of the little adjoint module of a simple Lie algebra $\, \mathfrak{g}\,$ of type $B, C\,$ or $ F.\,$ Let the skew invariants $\,(\bigwedge V_{\theta_s})^{\mathfrak{g}}\,$ be the exterior algebra $\,\bigwedge (T_{1},\cdots,T_{r_s}),\,$ $\,\{h_j\}_{j=1}^{r_s}\,$ the set of the special exponents of $\,(R_s^+)\,\check{},\,$ where $\,r_s\,$ is the dimension of the zero weight space of $V_{\theta_s},\,$ let $\,\text{ Der}^{\,n} \bigwedge V_{\theta_s}\,$ be the set of derivations of $\,\bigwedge V_{\theta_s}\,$ such that if $\,\psi \in \text{ Der}^{\,n} \bigwedge V_{\theta_s},$ $\,\psi: \bigwedge^m V_{\theta_s}\rightarrow \bigwedge^{n+m} V_{\theta_s}.\,$ 
The isotypic component of  $\,V_{\theta_s}\,$ is a free module over the subring of invariants $\,\bigwedge (T_{1},\cdots,T_{r_s-1})\,$ with basis vectors $\,a_j \in \bigwedge^{2h_j} V_{\theta_s}\,$ and $\,b_j \in \bigwedge^{2h_j-(2h_1-1)} V_{\theta_s},\,$ $\,1\leq j \leq r_s,\,$ where if $\,v\in V_{\theta_s},\,$ $\,a_j(v)= i_{\theta_s}(v)T_j,\,$ for some derivation $\,i_{\theta_s} \in \text{ Der}^{-1} \bigwedge V_{\theta_s}.\,$
 \end{conjecture}






\end{document}